\documentclass[a4paper, 12pt, reqno, english]{amsart}
\usepackage[utf8]{inputenc}
\usepackage[T1]{fontenc}
\usepackage{babel}

\usepackage[bookmarks = true, colorlinks, citecolor = blue, urlcolor = blue]{hyperref}
\usepackage[capitalize, nameinlink]{cleveref}
\usepackage[titletoc, toc, title]{appendix}

\hypersetup{pdfstartview=}

\usepackage{amssymb}
\usepackage{tikz}
\usetikzlibrary{calc}

\usepackage{dynkin-diagrams}

\numberwithin{equation}{section}
\numberwithin{figure}{section}

\theoremstyle{plain}
\newtheorem{theorem}{Theorem}[section]

\theoremstyle{plain}
\newtheorem*{theorem*}{Theorem}

\theoremstyle{plain}
\newtheorem{proposition}[theorem]{Proposition}

\theoremstyle{plain}
\newtheorem{lemma}[theorem]{Lemma}

\theoremstyle{plain}
\newtheorem{corollary}[theorem]{Corollary}

\theoremstyle{definition}
\newtheorem{definition}[theorem]{Definition}

\theoremstyle{definition}

\theoremstyle{definition}
\newtheorem{example}[theorem]{Example}

\theoremstyle{remark}
\newtheorem{remark}[theorem]{Remark}

\theoremstyle{definition}
\newtheorem{problem}[theorem]{Problem}

\theoremstyle{definition}
\newtheorem{conjecture}[theorem]{Conjecture}

\addto\extrasenglish{
  
}

\newcommand{\Uqb}{U_q(\mathfrak{b})}
\newcommand{\Uqg}{U_q(\mathfrak{g})}
\newcommand{\Uqh}{U_q(\mathfrak{h})}
\newcommand{\Uqk}{U_q(\mathfrak{k})}
\newcommand{\Uqn}{U_q(\mathfrak{n})}

\newcommand{\Uqnm}{U_q(\mathfrak{n}_-)}

\newcommand{\UqlS}{U_q(\mathfrak{l}_S)}
\newcommand{\Uqlk}{U_q(\mathfrak{l}_k)}

\newcommand{\CqG}{\mathbb{C}_q[G]}

\newcommand{\id}{\mathrm{id}}

\newcommand{\bbC}{\mathbb{C}}
\newcommand{\bbN}{\mathbb{N}}
\newcommand{\bbR}{\mathbb{R}}
\newcommand{\bbZ}{\mathbb{Z}}

\newcommand{\op}{\mathrm{op}}

\newcommand{\lieb}{\mathfrak{b}}
\newcommand{\liec}{\mathfrak{c}}
\newcommand{\lieg}{\mathfrak{g}}
\newcommand{\lieh}{\mathfrak{h}}
\newcommand{\liek}{\mathfrak{k}}
\newcommand{\liel}{\mathfrak{l}}
\newcommand{\lien}{\mathfrak{n}}
\newcommand{\liep}{\mathfrak{p}}
\newcommand{\lieu}{\mathfrak{u}}

\newcommand{\calA}{\mathcal{A}}
\newcommand{\calB}{\mathcal{B}}
\newcommand{\calC}{\mathcal{C}}
\newcommand{\calQ}{\mathcal{Q}}
\newcommand{\calR}{\mathcal{R}}

\newcommand{\schub}{U_q}
\newcommand{\tschub}{U'_q}

\newcommand{\bfM}{\mathbf{M}}

\newcommand{\roots}{\Phi}
\newcommand{\simpleroots}{\Pi}

\newcommand{\lspan}{\mathrm{span}}

\newcommand{\tmap}{\mathsf{T}}
\newcommand{\tmapS}{\mathsf{T}_S}
\newcommand{\pmap}{\mathsf{P}}

\newcommand{\qnil}{\lien^q}

\newcommand{\diff}{\mathrm{d}}
\newcommand{\del}{\partial}
\newcommand{\delbar}{\bar{\partial}}

\newcommand{\Cq}{\mathbb{C}_q}

\newcommand{\tang}{\mathcal{T}}

\begin{document}

\title[Equivariant quantizations and covariant differential calculi]{Equivariant quantizations of the positive nilradical and covariant differential calculi}

\author{Marco Matassa}

\address{OsloMet – Oslo Metropolitan University}

\email{marco.matassa@oslomet.no}

\begin{abstract}
Consider a decomposition $\lien = \lien_1 \oplus \cdots \oplus \lien_r$ of the positive nilradical of a complex semisimple Lie algebra of rank $r$, where each $\lien_k$ is a module under an appropriate Levi factor.
We show that this can be quantized as a finite-dimensional subspace $\qnil_k = \qnil_1 \oplus \cdots \oplus \qnil_r$ of the positive part of the quantized enveloping algebra, where each $\qnil_k$ is a module under the left adjoint action of a quantized Levi factor.
Furthermore, we show that $\bbC \oplus \qnil$ is a left coideal, with the possible exception of components corresponding to some exceptional Lie algebras.
Finally we use these quantizations to construct covariant first-order differential calculi on quantum flag manifolds, compatible in a certain sense with the decomposition above, which coincide with those introduced by Heckenberger-Kolb in the irreducible case.
\end{abstract}

\maketitle

\section{Introduction}

\subsection{Background}

Any Lie subalgebra $\liek \subset \lieg$ gives rise to a Hopf subalgebra $U(\liek) \subset U(\lieg)$ at the level of universal enveloping algebras. Conversely, any Hopf subalgebra of $U(\lieg)$ is of this form by the Milnor-Moore theorem (in characteristic zero).
The situation is drastically different for the \emph{quantized enveloping algebra} $\Uqg$, which is a Hopf algebra deformation of $U(\lieg)$, where $\lieg$ is a complex semisimple Lie algebra.
In this case it is well-known that $\Uqg$ admits fewer Hopf subalgebras than in the classical case, essentially only those corresponding to Dynkin subdiagrams.
Furthermore, even when a Hopf subalgebra $\Uqk \subset \Uqg$ is available, there is no obvious counterpart for the underlying Lie subalgebra $\liek \subset \lieg$.

The general philosophy to overcome this problem is that Hopf subalgebras should be replaced by \emph{coideal subalgebras}, either one-sided or two-sided. These reduce to Hopf subalgebras in the case of $U(\lieg)$, due to cocommutativity, but are more general in the case of $\Uqg$.
For example, this point of view has been used to develop the theory of quantum symmetric spaces, culminating with the results of Letzter \cite{letzter, letzter-review}, in which she constructs a family of one-sided coideals subalgebras $U_q(\lieg^\theta)$ reducing in the classical limit to $U(\lieg^\theta)$, where $\lieg^\theta$ is the fixed-point subalgebra with respect to a maximally split involution $\theta$.
In a related direction, Heckenberger and Schneider \cite{heckenberger-schneider} classify the right coideals subalgebras of the Borel part of $\Uqg$ containing the Cartan part (we review this and related results later in the introduction).
The situation is less clear for a quantum analogue of a Lie subalgebra within $\Uqg$.
Despite various definitions having been proposed over the years, such as \cite{majid-lie-algebras}, \cite{delius-gould} and \cite{generalized-an}, there is still no universally agreed-upon candidate.

In this paper we consider the problem of quantizing the positive nilradical $\lien$ of a complex semisimple Lie algebra $\lieg$ of rank $r$, together with a fixed decomposition $\lien = \lien_1 \oplus \cdots \oplus \lien_r$ which we are going to discuss shortly. The goal is to produce a \emph{finite-dimensional} subspace $\qnil = \qnil_1 \oplus \cdots \oplus \qnil_r$ of $\Uqn$, the positive part of the quantized enveloping algebra, in such a way that $\bbC \oplus \qnil$ should be a coideal (we discuss the precise formulation later).
This should be an \emph{equivariant} quantization, in the sense that each summand $\qnil_k$ should be a module under the left adjoint action of an appropriate quantized Levi factor, in parallel with the classical situation.
Similarly, the requirements on $\qnil$ being finite-dimensional and a coideal should be considered as replacement for the existence of a Lie algebra structure on $\lien$.

One important motivation for this construction comes from the theory of covariant differential calculi on quantum groups and homogeneous spaces, introduced by Woronowicz in \cite{woronowicz}.
As we are going to show, each quantization $\qnil$ gives rise to covariant (first-order) differential calculi on a class of quantum flag manifolds.
Furthermore, this assignment is compatible with the decomposition of $\qnil$, in an appropriate sense to be discussed later.

\subsection{Setting}

The classical setting we want to consider is as follows.
Fix a set $\simpleroots = \{ \alpha_1, \cdots, \alpha_r \}$ of simple roots of $\lieg$.
Then there is a one-to-one correspondence
\[
\alpha_{i_1} < \cdots < \alpha_{i_r}
\quad \Longleftrightarrow \quad
S_1 \subset \cdots \subset S_r
\]
between total orders on the simple roots and increasing sequences of proper subsets of $\simpleroots$.
The increasing sequence is determined by removing the simple roots in the given order, while the total order is determined by the removal of the simple roots from the sequence.
More precisely, we have $S_k = S_{k + 1} \backslash \{\alpha_{i_k}\}$ for $k \in \{ 1, \cdots, r \}$ and $S_{r + 1} = \simpleroots$.
Denote by $\liel_k$ the Levi factor corresponding to the subset $S_k \subset \simpleroots$ as above, for any $k \in \{ 1, \cdots, r \}$.
Corresponding to any such choice, we have a decomposition of the positive nilradical $\lien \subset \lieg$ of the form
\[
\lien = \lien_1 \oplus \cdots \oplus \lien_r.
\]
Here each summand $\lien_k$ arises as the positive nilradical of a parabolic subalgebra corresponding to the inclusion $S_k \subset S_{k + 1}$ (for more details on this, see \cref{sec:decompositions}).
Furthermore, each summand $\lien_k$ is a $\liel_k$-module and is $\bbN$-graded by the simple root $\alpha_{i_k}$.

On the quantum side, we consider the quantized enveloping algebra $\Uqg$, which is a Hopf algebra with generators $\{ K_i^{\pm 1}, E_i, F_i \}_{i = 1}^r$ and suitable relations.
We write $\Uqn$ for the positive part with generators $\{ E_i \}_{i = 1}^r$, while we write $\Uqb$ for the non-negative (or Borel) part with generators $\{ K_i^{\pm 1}, E_i \}_{i = 1}^r$.
We note that the latter is a Hopf subalgebra, while the former is not.
We can also find Hopf subalgebras $\Uqlk \subset \Uqg$ corresponding to each Levi factor $\liel_k$ as above (since these arise from Dynkin subdiagrams of $\lieg$).
Finally, writing $\triangleright$ for the left adjoint action of $\Uqg$ on itself, we consider the following problem.

\begin{problem}
\label{prob:quantization}
Let $\lieg$ be a complex semisimple Lie algebra of rank $r$.
For any increasing sequence $S_1 \subset \cdots \subset S_r$ as above, find a finite-dimensional subspace
\[
\qnil = \qnil_1 \oplus \cdots \oplus \qnil_r \subset \Uqn
\]
such that each summand $\qnil_k$ satisfies the following properties:
\begin{itemize}
\item $\qnil_k$ is an $U_q(\liel_k)$-module under $\triangleright$, corresponding to the $\liel_k$-module $\lien_k$,
\item each graded component $\qnil_{k, n}$ is contained in the $n$-fold product of $\qnil_{k, 1}$.
\end{itemize}
\end{problem}

The first condition simply says that each $\qnil_k$ gives a quantization of $\lien_k$ with its action by the Levi factor $\liel_k$.
On the other hand, the second condition should be considered as a weak replacement for the Lie algebra structure on $\lien_k$ (this condition is going to be refined shortly).
We refer to any $\qnil \subset \Uqn$ as above as an \emph{equivariant quantization} of $\lien$ (with its fixed decomposition).
We observe that one could formulate a related problem by replacing $\Uqn$ with $\Uqb$, the non-negative part of $\Uqg$ quantizing the Borel subalgebra $\lieb$.

As we are going to discuss, this problem is not too difficult to solve using general results on the representation theory of $\Uqg$.
On the other hand, not all such equivariant quantizations correspond to $\lien$ upon specialization to the classical case.
To avoid this issue and to remedy the lack of uniqueness, we should impose further conditions on $\qnil$.

We note that the classical property $\Delta(\lien) \subset \lien \otimes 1 + 1 \otimes \lien$ inside $U(\lien)$ can be rephrased as saying that $\bbC \oplus \lien$ is a left (and right) $U(\lien)$-coideal.
Similarly, we require that $\bbC \oplus \qnil$ should be a coideal of some sorts. The precise problem we consider is as follows.

\begin{problem}
\label{prob:coideal}
With notation as above, find an equivariant quantization $\qnil = \qnil_1 \oplus \cdots \oplus \qnil_r$ such that each subspace $\bbC \oplus \qnil_k$ is a left $\Uqb$-coideal, that is
\[
\Delta(\bbC \oplus \qnil_k) \subset \Uqb \otimes (\bbC \oplus \qnil_k), \quad
k \in \{ 1, \cdots, r \}.
\]
\end{problem}

We note that here the coideal condition is formulated in terms of $\Uqb$ instead of $\Uqn$, in other words we must also include the Cartan part.
This is related to the fact that $\Uqn$ is not a Hopf subalgebra, as mentioned above. More concretely, for its generators $E_i$ we have the coproduct $\Delta(E_i) = E_i \otimes 1 + K_i \otimes E_i$ (in our conventions).
This also motivates the choice of working with left coideals as opposed to right coideals.

Finally we describe a related setting that motivates the results of this paper.
Corresponding to any sequence $S_1 \subset \cdots \subset S_r$, we have the generalized flag manifolds
\[
G / P_{S_1} \to \cdots \to G / P_{S_r},
\]
where the arrows denote appropriate quotient maps. We note that $G / P_{S_1}$ corresponds to the full flag manifold $G / B$. This situation can quantized to give
\[
\Cq[G / P_{S_r}] \subset \cdots \subset \Cq[G / P_{S_1}],
\]
where each \emph{quantum flag manifold} $\Cq[G / P_{S_k}]$ is defined as the subalgebra of the quantized coordinate ring $\Cq[G]$ invariant under the quantized Levi factor $\Uqlk$.

We would like to construct covariant (first-order) differential calculi over quantum homogeneous spaces of this type.
We note that covariance here refers to a quantum version of a group $G$ acting on the differential forms of a homogeneous space $G / K$.
The results of \cite{heko-homogeneous} give a characterization of covariant first-order differential calculi over a certain class of quantum homogeneous spaces, which includes the quantum flag manifolds as above.
Such $n$-dimensional calculi over $\calB$ correspond to certain $(n + 1)$-subspaces of $\calB^\circ$, where the latter denotes the dual coalgebra.
As we are going to show, each equivariant quantization $\qnil$ has exactly the right properties to produce such first-order differential calculi.

\subsection{Results}

\cref{prob:quantization} can be solved by making use of some well-known facts on the representation theory of $\Uqg$, which lead to the following result.

\begin{proposition}
\label{prop:intro-quantization}
Let $\lieg$ be a complex semisimple Lie algebra of rank $r$.
Let $S_1 \subset \cdots \subset S_r$ be an increasing sequence of proper subsets of $\simpleroots$, with $\lien = \lien_1 \oplus \cdots \oplus \lien_r$ the corresponding decomposition of the positive nilradical.

Then there exists a finite-dimensional subspace
\[
\qnil = \qnil_1 \oplus \cdots \oplus \qnil_r \subset \Uqn
\]
such that each summand $\qnil_k$ satisfies the following properties:
\begin{itemize}
\item $\qnil_k$ is an $U_q(\liel_k)$-module under $\triangleright$, corresponding to the $\liel_k$-module $\lien_k$,
\item each graded component $\qnil_{k, n}$ is contained in the $n$-fold product of $\qnil_{k, 1}$.
\end{itemize}
\end{proposition}

Furthermore, each summand $\qnil_k$ can be realized as a subspace of a twisted quantum Schubert cell as in \cite{zwicknagl}, which is a modification of the quantum Schubert cell introduced by De Concini, Kac and Procesi in \cite{dckp}.
This implies that each degree-one component $\qnil_{k, 1}$ consists of quantum root vectors (for some appropriate reduced decomposition).

The equivariant quantization $\qnil = \qnil_1 \oplus \cdots \oplus \qnil_r$ from \cref{prop:intro-quantization} is not uniquely determined, in general. However this is notably the case when all the simple components are of classical type, that is corresponding to the series $A_r$, $B_r$, $C_r$ and $D_r$.
Therefore extra conditions need to be imposed to obtain uniqueness and properly recover the positive nilradical $\lien$ in the classical limit, as we are going to do now.

In \cref{prob:coideal} we impose the additional condition that $\bbC \oplus \qnil$ should be a a coideal. This problem appears to be harder to solve, and we are not able to do so in full generality here.
However we can solve it by making an additional assumption on $\lieg$, namely that certain exceptional Lie algebras do not appear within its simple components.

\begin{theorem}
\label{thm:intro-coideal}
Let $\lieg$ be a complex semisimple Lie algebra of rank $r$, which does not contain components of type $F_4$, $E_7$, $E_8$.
Let $S_1 \subset \cdots \subset S_r$ be an increasing sequence of proper subsets of $\simpleroots$, with $\lien = \lien_1 \oplus \cdots \oplus \lien_r$ the corresponding decomposition of the positive nilradical.

Then there exists a unique finite-dimensional subspace
\[
\qnil = \qnil_1 \oplus \cdots \oplus \qnil_r \subset \Uqn
\]
such that each summand $\qnil_k$ satisfies the following properties:
\begin{itemize}
\item $\qnil_k$ is an $U_q(\liel_k)$-module under $\triangleright$, corresponding to the $\liel_k$-module $\lien_k$,
\item each graded component $\qnil_{k, n}$ is contained in the $n$-fold product of $\qnil_{k, 1}$,
\item $\bbC \oplus \qnil_k$ is a left $\Uqb$-coideal, that is $\Delta(\bbC \oplus \qnil_k) \subset \Uqb \otimes (\bbC \oplus \qnil_k)$.
\end{itemize}
\end{theorem}

The result we prove is slightly stronger than stated, namely we can quantize any graded component $\qnil_{k, n}$ up to $n = 3$ as left a $\Uqb$-coideal, including in the cases $F_4, E_7$ and $E_8$.
Hence it remains to check that this holds for components of higher degrees, which becomes more challenging.
We note that the most complicated case, corresponding to $E_8$, would require this verification up to degree six (which is the largest coefficient of a simple root in its highest root).
This provides some further evidence for the following conjecture.

\begin{conjecture}
The results of \cref{thm:intro-coideal} hold for any finite-dimensional complex semisimple $\lieg$, without any restriction on the simple components.
\end{conjecture}

Next we discuss, given an equivariant quantization $\qnil$ as above, how to construct covariant first-order differential calculi on quantum flag manifolds.
As previously mentioned, the results of \cite{heko-homogeneous} give a way to produce such finite-dimensional calculi over a quantum homogeneous space in terms of certain subspaces of its dual coalgebra.
Consider the case of a quantum flag manifold $\Cq[G / P_{S_k}]$, defined by invariance under the Hopf subalgebra $\Uqlk$.
Then we obtain a differential calculus by producing a finite-dimensional subspace $T \subset \Uqg$ which contains the unit, is a right $\Uqg$-coideal, and is invariant under the left adjoint action of $\Uqlk$.

These are exactly the properties satisfied by $\bbC \oplus \qnil$, except that it is a left coideal.
We can easily pass to right coideals by using the quantum analogue of the Chevalley involution, which is a Hopf algebra isomorphism $\omega: \Uqg \to \Uqg^{\mathrm{cop}}$, where the latter denotes $\Uqg$ with opposite coproduct.
Furthermore we consider $\Uqg$ as a Hopf $*$-algebra, with $*$-structure corresponding to the compact real form.
Then we have the following result.

\begin{theorem}
\label{thm:intro-calculi}
Under the same assumptions of \cref{thm:intro-coideal}, consider the unique equivariant quantization $\qnil = \qnil_1 \oplus \cdots \oplus \qnil_r$ corresponding to $S_1 \subset \cdots \subset S_r$. Write
\[
\tang_- = \bbC \oplus \omega(\qnil), \quad
\tang_+ = \bbC \oplus \omega(\qnil)^*.
\]
Then we have the following:
\begin{enumerate}
\item $\tang_-$ and $\tang_+$ give left-covariant first-order differential calculi over the quantum group $\CqG$ and any quantum flag manifold $\Cq[G / P_{S_k}]$ with $k \in \{ 1, \cdots, r \}$,
\item if the quantum flag manifold $\Cq[G / P_{S_k}]$ is irreducible, these calculi coincide with those introduced by Heckenberger and Kolb in \cite{heko}.
\end{enumerate}
\end{theorem}

Here $\tang_-$ and $\tang_+$ should be thought of as corresponding to the holomorphic and anti-holomorphic parts of a differential calculus, since all flag manifolds are complex manifolds.
We recall that, in the case of a quantum \emph{irreducible} flag manifold $\Cq[G / P_S]$, Heckenberger and Kolb construct in \cite{heko} a canonical differential calculus that shares many features of its classical counterpart, including this type of decomposition.
Here canonical refers to the fact there exist only two irreducible left-covariant first-order calculi on $\Cq[G / P_S]$, corresponding to the holomorphic and anti-holomorphic parts as above, and one considers their direct sum.
This should be contrasted with the general situation of a quantum homogeneous space, where the choice of a differential calculus appears to be highly non-canonical.
For the quantum full flag manifolds of type $A$, generalizations of the Heckenberger-Kolb calculi were recently constructed in \cite{reamonn-ar} (where the the extension to higher-order forms is also studied).

We conclude this part with some remarks on covariant differential calculi on quantum groups and quantum homogeneous spaces, inspired by \cref{thm:intro-calculi}.
This result suggests that the choice of a differential calculus on $\CqG$ should be \emph{adapted} to a sequence
\[
\Cq[G / P_{S_r}] \subset \cdots \subset \Cq[G / P_{S_1}]
\]
of quantum flag manifolds, corresponding to the increasing sequence $S_1 \subset \cdots \subset S_r$ of $\simpleroots$.
This in general will not be enough to single out a unique candidate for a differential calculus, since it should be possible to twist the previous construction by Cartan elements.
This would correspond to considering \cref{prob:quantization} and \cref{prob:coideal} within the Borel part $\Uqb$, as opposed to the positive part $\Uqn$.
However this difference should disappear at the level of quantum flag manifolds, since each quantized Levi factor $\Uqlk$ contains the Cartan part.
We expect that this could be leveraged to gain uniqueness in this setting.

\subsection{Existing literature}

We now come back to the first part of the introduction, by giving a quick overview of some related results existing in the literature.

The problem of classifying (right) coideal \emph{subalgebras} of $\Uqb$ containing $\Uqh$ has been studied extensively. This was first solved for $A_r$ in \cite{coideal-Ar}, then later on for $B_r$ in \cite{coideal-Br} and for $G_2$ in \cite{coideal-G2}, before a general classification was obtained by Heckenberger and Schneider in \cite{heckenberger-schneider}.
The result of their classification is that such coideals are in one-to-one correspondence with elements of the Weyl group $W$ by the map $w \mapsto \schub(w) \Uqh$, where $\schub(w)$ denotes the quantum Schubert cell corresponding to $w$ as above and $\Uqh$ is the Cartan part.
Furthermore this map is order-preserving, where $W$ is considered with the Duflo order and the coideal subalgebras are ordered by inclusion.
We should also mention \cite{heckenberger-kolb}, where coideals of $\Uqb$ satisfying a more general condition are also classified.

Comparing the results of our paper to those cited above, we see two main differences: 1) we are concerned with finite-dimensional coideals as opposed to coideal subalgebras, 2) we consider such coideals together with the action of appropriate quantized Levi factors.

Another result that should be mentioned in this context is the determination of the locally-finite part of $\Uqg$ under the left adjoint action, due to Joseph and Letzter \cite{joseph-letzter}.
They show that it consists of finite-dimensional blocks $B(\lambda)$, for each dominant weight $\lambda$, which are invariant under the left adjoint action of $\Uqg$.
Each block $B(\lambda)$ is a left $\Uqg$-coideal but is not irreducible in general (it is isomorphic to $V(\lambda)^* \otimes V(\lambda)$).
It appears to be non-trivial to use the results of \cite{joseph-letzter} for the problems considered in our paper, since we are interested in the locally-finite part of $\Uqn$ under $\UqlS$, for various choices of $S \subset \simpleroots$.

\subsection{Organization}

We conclude this introduction by outlining the organization of this paper.
In \cref{sec:classical-preliminaries} we review some classical results, in particular regarding the decompositions $\lien = \lien_1 \oplus \cdots \oplus \lien_r$.
Similarly, in \cref{sec:quantum-preliminaries} we review some of the corresponding preliminaries in the quantum case.
In \cref{sec:equivariant-quantizations} we prove the existence of equivariant quantizations as in \cref{prob:quantization}, and in \cref{sec:schubert-cells} we discuss the relation between this construction and quantum Schubert cells, or more precisely their twisted versions.
In \cref{sec:quasitriangular} we discuss some general results on quasitriangular Hopf algebras that are used in the next section.
In \cref{sec:coideal} we prove that there is a unique choice of $\qnil$ in such a way that $\bbC \oplus \qnil$ is a left $\Uqb$-coideal, under the previously discussed assumptions on the simple components, solving \cref{prob:coideal} in this case.
In \cref{sec:example-G2} we illustrate these results in the case of the exceptional Lie algebra $G_2$.
In \cref{sec:covariant-calculi} we discuss how to obtain covariant calculi on quantum flag manifold from the equivariant quantizations introduced here.
Finally in \cref{sec:appendix-computations} we prove a certain result on triple tensor products, which is needed in the proof of the coideal property.

\subsection*{Acknowledgments}

We are grateful to István Heckenberger for some interesting discussions related to the contents of this paper.
We would also like to thank Kenny De Commer and Réamonn Ó Buachalla for their comments and questions.

\section{Classical preliminaries}
\label{sec:classical-preliminaries}

In this section we recall some known results on parabolic subalgebras and decompositions in the Weyl group.
We also discuss in detail how to obtain the decomposition $\lien_1 \oplus \cdots \oplus \lien_r$ corresponding to $S_1 \subset \cdots \subset S_r$ mentioned in the introduction.

\subsection{Parabolic subalgebras}

Let $\lieg$ be a complex semisimple Lie algebra.
A choice of Cartan subalgebra $\lieh \subset \lieg$ determines a root system $\roots \subseteq \lieh^*$.
This can be decomposed into positive and negative roots as $\roots = \roots^+ \sqcup \roots^-$.
For any root $\alpha \in \roots$ we denote by $\lieg_\alpha \subset \lieg$ the corresponding root space.
Write $\simpleroots = \{ \alpha_1, \cdots, \alpha_r \}$ for the set of simple roots, so that $\lieg$ has rank $r$.

Any subset $S \subseteq \simpleroots$ determines a \emph{parabolic subalgebra} $\liep_S$ as follows. First set
\[
\roots(\liel_S) := \lspan(S) \cap \roots, \quad
\roots(\lien_S) := \roots^+ \backslash \roots(\liel_S).
\]
In terms of these roots we define
\[
\liel_S := \lieh \oplus \bigoplus_{\alpha \in \roots(\liel_S)} \lieg_\alpha, \quad
\lien_S := \bigoplus_{\alpha \in \roots(\lien_S)} \lieg_\alpha, \quad
\liep_S := \liel_S \oplus \lien_S.
\]
Then $\liel_S$, $\lien_S$ and $\liep_S$ are Lie subalgebras of $\lieg$.
We refer to $\liel_S$ as the \emph{Levi factor} and to $\lien_S$ as the (positive) \emph{nilradical} of the parabolic subalgebra $\liep_S$.

\subsection{Decompositions}
\label{sec:decompositions}

Consider an increasing sequence
\[
S_1 \subset S_2 \subset \cdots \subset S_{r - 1} \subset S_r \subset \simpleroots
\]
of proper subsets of the simple roots $\simpleroots$.
We have that each subset $S_k$ has cardinality $k - 1$, since $|\Pi| = r$.
For convenience we also write $S_{r + 1} = \simpleroots$ in the following.
Let us adopt the shorthand notation $\liel_k := \liel_{S_k}$ for the Levi factor corresponding to the subset $S_k \subseteq \simpleroots$.
Then we have an increasing sequence of reductive Lie algebras
\begin{equation}
\label{eq:levi-factor-inclusion}
\lieh = \liel_1 \subset \liel_2 \subset \cdots \subset \liel_r \subset \liel_{r + 1} = \lieg.
\end{equation}

Let $\lieg_k := [\liel_{S_k}, \liel_{S_k}]$ be the semisimple part of the Levi factor $\liel_{S_k}$.
In other words, $\lieg_k \subseteq \lieg$ is the semisimple Lie algebra obtained from the Dynkin subdiagram corresponding to $S_k \subseteq \simpleroots$.
Then we have an increasing sequence of complex semisimple Lie algebras
\begin{equation}
\label{eq:lie-algebra-inclusion}
\{0\} = \lieg_1 \subset \lieg_2 \subset \cdots \subset \lieg_r \subset \lieg_{r + 1} = \lieg.
\end{equation}

Since the set $S_{k + 1}$ can be identified with the simple roots of $\lieg_{k + 1}$, we see that the inclusion $S_k \subset S_{k + 1}$ defines a parabolic subalgebra of $\lieg_{k + 1}$.
The semisimple part of its Levi factor is exactly $\lieg_k$.
Then we denote by $\lien_k \subset \lieg_{k + 1}$ the corresponding positive nilradical (note the shift by one in the notation).
We also observe that, in general, this is not the same as the positive nilradical corresponding the inclusion $S_k \subset \simpleroots$ for the Lie algebra $\lieg$.

Then, with notation as above, we have the decomposition
\[
\lien = \lien_1 \oplus \cdots \oplus \lien_r
\]
corresponding to the sequence $S_1 \subset \cdots \subset S_r$.
Each nilradical $\lien_k$ is a $\liel_k$-module, so we have an equivariant decomposition with respect to these actions.

\begin{example}
To illustrate the setting above, consider $\lieg = A_3 = \mathfrak{sl}_4$ with the usual numbering for the simple roots.
Consider the sequence of proper subsets given by
\[
S_1 = \emptyset, \quad
S_2 = \{ \alpha_1 \}, \quad
S_3 = \{ \alpha_1, \alpha_2 \}.
\]
Then the Lie subalgebras $\liel_k$ and $\lieg_k$ are given by
\begin{align*}
\liel_1 & = \lieh, \quad &
\liel_2 & = \lieh \oplus \lspan \{ e_{\alpha_1}, f_{\alpha_1} \}, \quad &
\liel_3 & = \lieh \oplus \lspan \{ e_{\alpha_1}, e_{\alpha_2}, e_{\alpha_1 + \alpha_2}, f_{\alpha_1}, f_{\alpha_2}, f_{\alpha_1 + \alpha_2} \}, \\
\lieg_1 & = \{ 0 \}, \quad &
\lieg_2 & = \lspan \{ h_{\alpha_1}, e_{\alpha_1}, f_{\alpha_1} \}, \quad &
\lieg_3 & = \lspan \{ h_{\alpha_1}, h_{\alpha_2}, e_{\alpha_1}, e_{\alpha_2}, e_{\alpha_1 + \alpha_2}, f_{\alpha_1}, f_{\alpha_2}, f_{\alpha_1 + \alpha_2} \}.
\end{align*}
The corresponding decomposition $\lien = \lien_1 \oplus \lien_2 \oplus \lien_3$ is given by
\[
\lien_1 = \lspan \{ e_{\alpha_1} \}, \quad
\lien_2 = \lspan \{ e_{\alpha_2}, e_{\alpha_1 + \alpha_2} \}, \quad
\lien_3 = \lspan \{ e_{\alpha_3}, e_{\alpha_2 + \alpha_3}, e_{\alpha_1 + \alpha_2 + \alpha_3} \}.
\]
\end{example}

\subsection{Gradings}

As observed in the introduction, an increasing sequence $S_1 \subset \cdots \subset S_r$ as above corresponds to a total order $\alpha_{i_1} < \cdots < \alpha_{i_r}$ on the simple roots.
This is determined by the equality $S_k = S_{k + 1} \backslash \{ \alpha_{i_k} \}$ for $k \in \{ 1, \cdots, r \}$, where we write $S_{r + 1} = \simpleroots$.

Consider the corresponding decomposition $\lien = \lien_1 \oplus \cdots \oplus \lien_r$.
Then each summand $\lien_k$ is $\bbN$-graded by the simple root $\alpha_{i_k}$. Here a positive root $\alpha \in \roots^+(\lien_k)$ has degree $n$ if the simple root $\alpha_{i_k}$ appears in $\alpha$ with coefficient $n$.
We write this decomposition as
\[
\lien_k = \bigoplus_{n = 1}^{N_k} \lien_{k, n},
\]
where $\lien_{k, n}$ denotes the component of degree $n$.
We note that here $N_k$ corresponds to the maximal coefficient of $\alpha_{i_k}$ appearing in any root of $\lien_k$.

Furthermore, each $\lien_{k, n}$ is an $\liel_k$-module with respect to the adjoint action, since the simple root $\alpha_{i_k}$ does not appear in $\roots(\liel_k)$.
The following result gives more properties concerning these modules, which are going to be used in the quantum setting as well.

\begin{proposition}
\label{prop:classical-module}
The $\liel_k$-modules $\lien_{k, n}$ satisfy the following properties:
\begin{enumerate}
\item Each $\lien_{k, n}$ is a simple $\liel_k$-module.
\item For $n \leq N_k$, the simple component $\lien_{k, n}$ occurs in $\lien_{k, 1}^{\otimes n}$.
\item The tensor product $\lien_{k, m} \otimes \lien_{k, n}$ has a multiplicity-free decomposition.
\end{enumerate}
\end{proposition}

\begin{proof}
(1) This follows by a general result of Kostant on root systems of Levi factors (of which we consider special cases here), see in particular \cite[Theorem 1.9]{kostant}.

(2) Translating \cite[Theorem 2.3]{kostant} into our notation, we have $[\lien_{k, a}, \lien_{k, b}] = \lien_{k, a + b}$ provided $a + b \leq N_k$. This fact easily implies the claim.

(3) This is a general result on the representation theory of semisimple Lie algebras.
Each component $V(\nu)$ occurring in $V(\lambda) \otimes V(\mu)$ is of the form $\nu = \lambda + \mu'$, with $\mu'$ a weight of $V(\mu)$, and moreover for its multiplicity we have $m^\nu_{\lambda, \mu} \leq \dim V(\mu)_{\mu'}$.
Then the result follows by observing that all weights of the modules $\lien_{k, n}$ are one-dimensional, since they are roots.
\end{proof}

\subsection{Weyl group}

Let $W$ be the Weyl group of the complex semisimple Lie algebra $\lieg$ of rank $r$, generated by the simple reflections $s_1, \cdots, s_r$.
We write $\ell$ for the word-length function on $W$ and denote by $w_0$ be the longest word of $W$.

Let $w = s_{i_1} \cdots s_{i_n}$ be a reduced decomposition of a Weyl group element. Then it is well-known that we obtain distinct positive roots of $\lieg$ by
\[
\beta_k = s_{i_1} \cdots s_{i_{k - 1}} (\alpha_{i_k}), \quad
k = 1, \cdots, n.
\]
In particular, fixing a reduced decomposition $w_0 = s_{i_1} \cdots s_{i_N}$ of the longest word,  we obtain an enumeration $\beta_1, \cdots, \beta_N$ of the positive roots of $\lieg$.
This enumeration satisfies the following convexity property, see \cite{papi}: if $i < j$ and $\beta_k = \beta_i + \beta_j$ then $i < k < j$.

Let $S \subseteq \simpleroots$ be any subset of the simple roots of $\lieg$.
This corresponds to a parabolic subalgebra $\liep_S = \liel_S \oplus \lien_S$.
Let $w_{0, S}$ be the longest word for the Weyl group of the semisimple part of $\liel_S$.
Then we define the \emph{parabolic element} $w_S = w_{0, S} w_0$ corresponding to $S$.

We recall some properties corresponding to these elements.

\begin{lemma}
Consider a parabolic element $w_S$ as defined above. Then:
\begin{enumerate}
\item We have $w_0 = w_{0, S} w_S$ and $\ell(w_0) = \ell(w_{0, S}) + \ell(w_S)$.
\item If we write $\ell(w_{0, S}) = M$ then
\[
\beta_1, \cdots, \beta_M \in \roots^+(\liel_S), \quad
\beta_{M + 1}, \cdots, \beta_N \in \roots^+(\lien_S).
\]
\item The positive roots $\roots^+(\lien_S)$ are invariant under $w_{0, S}$.
\item The positive roots corresponding to $w_S$ are exactly $\roots^+(\lien_S)$.
\end{enumerate}
\end{lemma}

\begin{proof}
(1) Clearly $w_0 = w_{0, S} w_S$ since $w_{0, S}$ is involutive, being the longest word of a Weyl group.
For the statement on the length see for instance \cite[Section 1.8]{humphreys}.

(2) Clearly $\beta_1, \cdots, \beta_M \in \roots^+(\liel_S)$ since $w_{0, S}$ is the longest word for the semisimple part of $\liel_S$.
The result follows since $w_0$ gives an enumeration of the positive roots of $\lieg$ and we have the disjoint union $\roots^+(\lieg) = \roots^+(\liel_S) \sqcup \roots^+(\lien_S)$.

(3) This follows from the fact that $\lien_S$ is invariant under the adjoint action of $\liel_S$.

(4) Fix a reduced decomposition of $w_S$ and denote by $\{ \gamma_k \}_{k = 1}^{N - M}$ the corresponding positive roots.
Let $\{ \beta_k \}_{k = 1}^N$ be the positive roots corresponding to $w_0 = w_{0, S} w_S$ as above.
We note that $\beta_{M + k} = w_{0, S} (\gamma_k)$ for $k = 1, \cdots, N - M$.
We have $\beta_{M + 1}, \cdots, \beta_N \in \roots^+(\lien_S)$ by (2).
Since $\roots^+(\lien_S)$ is invariant under $w_{0, S}$ by (3), the result follows.
\end{proof}

\section{Quantum preliminaries}
\label{sec:quantum-preliminaries}

In this section we briefly recall various results concerning the quantized enveloping algebra $\Uqg$, with the books  \cite{jantzen} and \cite{brown-goodearl} as our main references.
In particular this includes the braid group action given by the Lusztig automorphisms $T_i$, the corresponding quantum root vectors, and the $q$-commutation relations among these.

\subsection{Quantized enveloping algebras}

Let $\lieg$ be a complex semisimple Lie algebra of rank $r$.
The \emph{quantized enveloping algebra} $\Uqg$ is the algebra with generators $\{ K_i^{\pm 1}, E_i, F_i \}_{i = 1}^r$ and relations as in \cite[Section 4.3]{jantzen} or \cite[Section I.6.3]{brown-goodearl}.

It becomes a Hopf algebra by the following coproduct
\[
\Delta(K_i) = K_i \otimes K_i, \quad
\Delta(E_i) = E_i \otimes 1 + K_i \otimes E_i, \quad
\Delta(F_i) = F_i \otimes K_i^{-1} + 1 \otimes F_i.
\]
The corresponding antipode is
\[
S(K_i) = K_i^{-1}, \quad
S(E_i) = - K_i^{-1} E_i, \quad
S(F_i) = - F_i K_i.
\]
As for any Hopf algebra, we have the left \emph{adjoint action} of $\Uqg$ on itself defined by
\[
x \triangleright y := x_{(1)} y S(x_{(2)}).
\]
Here we use Sweedler's notation $\Delta(x) = x_{(1)} \otimes x_{(2)}$ for the coproduct.

We note that $\Uqg$ can be defined over the field of fractions $\bbC(q)$ of an indeterminate $q$, or over $\bbC$ with $q \in \bbC$ under some restrictions.
For the purpose of this paper the first version is slightly more convenient, but we can as well work with the second version under the assumption that $q$ is transcendental.
In the last section we are also going to consider $*$-structures, in which case $q^* = q$ in the first version and $q \in \bbR$ in the second one.

Let $\liel_S \subseteq \lieg$ be the Levi subalgebra corresponding to a subset $S \subseteq \simpleroots$.
Then the \emph{quantized Levi factor} $\UqlS$ can be defined analogously to $\Uqg$, with the difference that in this case there are some central elements in the Cartan part.
For our purposes, it is more convenient to define it as the Hopf subalgebra $\UqlS \subseteq \Uqg$  generated by the elements $\{ K_i^{\pm 1} \}_{i = 1}^r \cup \{ E_i, F_i \}_{i \in S}$.
Clearly the left adjoint action restricts to an action of $\UqlS$ on $\Uqg$.

The representation theory of $\Uqg$ parallels the classical situation, since
the $\Uqg$-modules (of type 1) are in one-to-one correspondence with the $\lieg$-modules.
More precisely, in both cases the (finite-dimensional) simple modules are labeled by their highest weights, and similarly when comparing $\UqlS$ and $\liel_S$.
Furthermore, the tensor product of $\Uqg$-modules or $\UqlS$-modules decomposes into simple components as in the classical case.

Finally we recall that the category of finite-dimensional $\Uqg$-modules is braided.
The choice of braiding is not quite unique, and the one we consider can be characterized by
\[
\hat{R}_{V, W}(v_\mu \otimes w_\nu) = q^{(\mu, \nu)} w_\nu \otimes v_\mu + \sum_i w_i \otimes v_i,
\]
where for the additional terms $w_i$ has weight higher than $w_\nu$ and $v_i$ has weight lower than $v_\mu$.
This particular choice is going to play a role later in \cref{sec:coideal}.

\subsection{Braid group action}

Let $B_W$ be the braid group corresponding to the Weyl group $W$.
Concretely, it has the same relations as $W$ except for the relations $s_i^2 = 1$ for the simple reflections.
There is an action of $B_W$ on $\Uqg$ by algebra automorphisms $T_i$ with $i \in \{ 1, \cdots, r \}$, usually called \emph{Lusztig automorphisms}, see for instance \cite[Section I.6.7]{brown-goodearl}.

Given any reduced decomposition $w = s_{i_1} \cdots s_{i_n} \in W$ we define
\[
T_w = T_{i_1} \cdots T_{i_n}.
\]
This algebra automorphism does not depend on the chosen reduced decomposition of $w$, since the $T_i$'s satisfy the braid relations.
If $X_\mu \in \Uqg$ has weight $\mu$ then $T_w(X_\mu)$ has weight $w(\mu)$.

There are many possible conventions for the Lusztig automorphisms $T_i$.
For definiteness, we adopt the conventions from \cite[Section I.6.7]{brown-goodearl}. In particular we have
\begin{equation}
\label{eq:lusztig-relations}
T_i(K_\alpha) = K_{s_i(\alpha)}, \quad
T_i(E_i) = - F_i K_i, \quad
T_i(F_i) = - K_i^{-1} E_i.
\end{equation}

\subsection{Quantum root vectors}

Fix a reduced decomposition $w_0 = s_{a_1} \cdots s_{a_N}$ of the longest word of the Weyl group of $\lieg$.
Then the \emph{quantum root vectors} $E_{\beta_i}$ are defined by
\[
E_{\beta_i} := T_{a_1} \cdots T_{a_{i - 1}}(E_{a_i}), \quad i \in \{ 1, \cdots, N \}.
\]
Note that they depend on the choice of reduced decomposition of $w_0$.
The quantum root vectors $F_{\beta_i}$ corresponding to the negative roots are defined similarly.

The quantum root vectors can be used to construct a basis of $\Uqn$.
Let us write $\mathbf{M} = (m_1, \cdots, m_N)$ for an $N$-tuple of non-negative integers. Then the elements
\[
E^{\mathbf{M}} := E_{\beta_1}^{(m_1)} \cdots E_{\beta_N}^{(m_N)}, \quad m_1, \cdots, m_N \geq 0,
\]
give the corresponding basis, see for instance \cite[Section I.6.8]{brown-goodearl}.

The commutation relations between quantum root vectors are significantly more complicated than in the classical case.
The following general result can be found for instance in \cite[Section I.6.10]{brown-goodearl}, which we are going to use at various points later on.

\begin{proposition}
\label{prop:q-commutation-relations}
For $i < j$ we have
\[
E_{\beta_i} E_{\beta_j} - q^{(\beta_i, \beta_j)} E_{\beta_j} E_{\beta_i} = \sum_{k_{i + 1}, \cdots, k_{j - 1}} c_{k_{i + 1}, \cdots, k_{j - 1}} E_{\beta_{i + 1}}^{k_{i + 1}} \cdots E_{\beta_{j - 1}}^{k_{j - 1}},
\]
where the coefficients $c_{k_{i + 1}, \cdots, k_{j - 1}}$ are certain Laurent polynomials in $q$.
\end{proposition}

We refer to this result as the \emph{$q$-commutation relations} for the quantum root vectors.
For any weight elements $X_\mu, Y_\nu \in \Uqg$, we can define the \emph{q-commutator} by
\[
[X_\mu, Y_\nu]_q := X_\mu Y_\nu - q^{(\mu, \nu)} Y_\nu X_\mu.
\]
Then the $q$-commutation relations express the $q$-commutator $[E_{\beta_i}, E_{\beta_j}]_q$ as a linear combination of elements $E_{\beta_{i + 1}}^{k_{i + 1}} \cdots E_{\beta_{j - 1}}^{k_{j - 1}}$ ranging from $i + 1$ to $j - 1$.

\begin{remark}
\label{rem:lusztig-automorphisms}
In the literature one finds various slightly different definitions of the Lusztig automorphisms. For instance, the $T_i$'s defined in \cite[Section 8.14]{jantzen} and \cite[Section I.6.7]{brown-goodearl} both satisfy the relations \eqref{eq:lusztig-relations} but differ by signs for $T_i(E_j)$ and $T_i(F_j)$ when $i \neq j$.

More generally, consider two different choices $T_i$ and $\tilde{T}_i$ which are related in the following way: for any weight vector $X_\mu$ of weight $\mu$, we have $\tilde{T}_i(X_\mu) = t_{i, \mu} T_i(X_\mu)$, where the coefficient $t_{i, \mu}$ depends only of the weight of $X_\mu$.
Then, writing $E_{\beta_k}$ and $\tilde{E}_{\beta_k}$ for the quantum root vectors for these two choices, we find that they only differ by an overall factor.
This implies that the form of the commutation relations is the same in both cases, since the overall factor on the left-hand side can be absorbed into the coefficients on the right-hand side.
\end{remark}

\section{Equivariant quantizations of the nilradical}
\label{sec:equivariant-quantizations}

With the proper preliminaries in place, we can now consider \cref{prob:quantization} as mentioned in the introduction.
We want to construct a quantum analogue of $\lien$ inside $\Uqn$ with a fixed decomposition $\lien = \lien_1 \oplus \cdots \oplus \lien_r$, corresponding to an increasing sequence $S_1 \subset \cdots \subset S_r$ of $\simpleroots$.
We show that this can be obtained rather easily by representation-theoretic arguments.
The given construction is not unique in general, but it is notably so in the case of the classical series $A_r$, $B_r$, $C_r$ and $D_r$.
Finally we discuss the classical situation in more detail, in order to motivate some of the requirements we are going to impose in later sections.

\subsection{Existence of quantizations}

Consider an increasing sequence $S_1 \subset \cdots \subset S_r$ of proper subsets of $\simpleroots$ and let $\alpha_{i_1} < \cdots < \alpha_{i_r}$ be the corresponding total order.
Then, as discussed in \cref{sec:decompositions}, from this data we obtain a decomposition of the positive nilradical
\[
\lien = \lien_1 \oplus \cdots \oplus \lien_r,
\]
where each $\lien_k$ is $\bbN$-graded by $\alpha_{i_k}$ and a $\liel_k$-module.

We look for a quantization of this decomposition in the following sense.

\begin{definition}
\label{def:equivariant-quantization}
An \emph{equivariant quantization} of the positive nilradical $\lien$, together with a decomposition corresponding to $S_1 \subset \cdots \subset S_r$, is a finite-dimensional subspace
\[
\qnil = \qnil_1 \oplus \cdots \oplus \qnil_r \subset \Uqn
\]
such that each summand $\qnil_k$ satisfies the following properties:
\begin{itemize}
\item each $\qnil_k$ is a $\Uqlk$-module under $\triangleright$, corresponding to the $\liel_k$-module $\lien_k$,
\item each graded component $\qnil_{k, n}$ is contained in the $n$-fold product of $\qnil_{k, 1}$.
\end{itemize}
\end{definition}

We remark that the grading of each module $\qnil_k$ is with respect to the simple root $\alpha_{i_k}$, using the fact that $\Uqn$ is graded by the root lattice.

Some comments on the definition.
In the first condition we make use of the fact that $\Uqg$-modules (of type 1) are in one-to-one correspondence with $\lieg$-modules, and similarly for $\UqlS$ and $\liel_S$.
The second condition is a weak replacement for the existence of a Lie algebra structure on $\lien_k$, which can be formulated in this way since $[\lien_{k, a}, \lien_{k, b}] = \lien_{k, a + b}$ by \cref{prop:classical-module}.

We begin by investigating the degree-one components of the decomposition.

\begin{lemma}
\label{lem:action-degree-one}
Suppose $\qnil_{k, 1} \subset \Uqn$ is an $\Uqlk$-module under the left adjoint action, corresponding to the $\liel_k$-module $\lien_{k, 1}$.
Then we must have $\qnil_{k, 1} = \Uqlk \triangleright E_{i_k}$.
\end{lemma}

\begin{proof}
The degree-one component $\lien_{k, 1}$ is a simple $\liel_k$-module by \cref{prop:classical-module}.
It is easy to see that it is generated by the lowest weight vector $e_{i_k}$, where $\alpha_{i_k}$ is the simple root with respect to which $\lien_k$ is $\bbN$-graded.
Similarly, the simple $\Uqlk$-module $\qnil_{k, 1} \subset \Uqn$ must be generated by a lowest weight vector of weight $\alpha_{i_k}$. The only elements of $\Uqn$ of weight $\alpha_{i_k}$ are multiples of the generator $E_{i_k}$.
In our conventions we have
\[
F_i \triangleright X = F_i X K_i - X F_i K_i = [F_i, X] K_i.
\]
This shows that $F_i \triangleright E_{i_k} = 0$ for any $F_i \in \Uqlk$, since $\alpha_{i_k}$ is not a root of the Levi factor $\liel_k$.
Therefore $E_{i_k}$ is a lowest weight vector and generates a simple $\Uqlk$-module corresponding to $\lien_{k, 1}$.
Finally we note that $\Uqlk \triangleright E_{i_k} \subset \Uqn$, since it suffices to take the positive part of $\Uqlk$ to generate this module, as $E_{i_k}$ is a lowest weight vector.
\end{proof}

This shows that the degree-one components $\qnil_{1, 1}, \cdots, \qnil_{r, 1}$ exist and are uniquely determined.
They are going to be the building blocks for the next result.

\begin{proposition}
\label{prop:equivariant-quantizations}
Let $\lieg$ be a complex semisimple Lie algebra of rank $r$.
Let $S_1 \subset \cdots \subset S_r$ be an increasing sequence of proper subsets of $\simpleroots$, with $\lien = \lien_1 \oplus \cdots \oplus \lien_r$ the corresponding decomposition of the positive nilradical.

Then there exists a finite-dimensional subspace
\[
\qnil = \qnil_1 \oplus \cdots \oplus \qnil_r \subset \Uqn
\]
such that each summand $\qnil_k$ satisfies the following properties:
\begin{itemize}
\item $\qnil_k$ is an $U_q(\liel_k)$-module under $\triangleright$, corresponding to the $\liel_k$-module $\lien_k$,
\item each graded component $\qnil_{k, n}$ is contained in the $n$-fold product of $\qnil_{k, 1}$.
\end{itemize}
\end{proposition}

\begin{proof}
In \cref{lem:action-degree-one} we proved that, for each $k \in \{ 1, \cdots, r \}$, the degree-one component $\qnil_{k, 1} \subset \Uqn$ exists and is uniquely determined.
The components $\qnil_{k, n}$ of degree $n > 1$ can be constructed as follows (but not in a unique way, in general).

Consider the $n$-fold product $\qnil_{k, 1} \cdots \qnil_{k, 1} \subset \Uqn$.
This is an $U_q(\liel_k)$-module and a quotient of $(\qnil_{k, 1})^{\otimes n}$ by the multiplication map in $\Uqn$.
We know that the $\liel_k$-module $\lien_{k, n}$ appears as a simple component of $\lien_{k, 1}^{\otimes n}$ by \cref{prop:classical-module}, hence the same is true for the corresponding $U_q(\liel_k)$-module. Then we can pick a simple component corresponding to the $\liel_k$-module $\lien_{k, n}$ within $(\qnil_{k, 1})^{\otimes n}$, and denote by $\qnil_{k, n}$ its image in $\Uqn$.
At least one such component is going to be non-zero under the multiplication map, by specialization at $q = 1$.
\end{proof}

We have the following simple result concerning the uniqueness of this construction.

\begin{corollary}
\label{cor:uniqueness-quantization}
Consider an equivariant quantization $\qnil = \qnil_1 \oplus \cdots \oplus \qnil_r$ as above.
\begin{enumerate}
\item The components $\qnil_{k, 1}$ and $\qnil_{k, 2}$ are uniquely determined for $k \in \{ 1, \cdots, r \}$.
\item For the classical series $A_r$, $B_r$, $C_r$, $D_r$ the quantization is uniquely determined.
\end{enumerate}
\end{corollary}

\begin{proof}
(1) We have already seen that each $\qnil_{k, 1}$ is uniquely determined in \cref{lem:action-degree-one}.
For the component of degree two, observe that $\qnil_{k, 1} \otimes \qnil_{k, 1}$ admits a multiplicity-free decomposition by \cref{prop:classical-module}, which implies that there is only one component of this type.

(2) It is easily checked that the classical series $A_r$, $B_r$, $C_r$, $D_r$ have the property that their highest roots contain each simple root with coefficient at most two. This means that the $\Uqlk$-module $\qnil_k$ has at most two components, for any $k \in \{ 1, \cdots, r \}$ and for any sequence $S_1 \subset \cdots \subset S_r$.
Then the uniqueness from (1) gives the result.
\end{proof}

This uniqueness does not extend to components of higher degree, as we discuss in the next subsection.
For this reason we need further conditions to obtain a reasonable analogue of $\lien = \lien_1 \oplus \cdots \oplus \lien_r$ for the exceptional Lie algebras, which we discuss later.

\subsection{Comments on the result}
\label{subsec:comments-result}

Here we want to illustrate some shortcomings of the construction given in the proof of \cref{prop:equivariant-quantizations}, and motivate some of the developments that appear in the next sections.
We do this by specializing the construction to the classical case, that is for the universal enveloping algebra $U(\lieg)$ corresponding to $\lieg$.

Recall that $U(\lieg)$ is a Hopf algebra with coproduct determined by
\[
\Delta(x) = x \otimes 1 + 1 \otimes x, \quad x \in \lieg.
\]
In particular, this is the case for any of the summands appearing in the decomposition of the positive nilradical $\lien = \lien_1 \oplus \cdots \oplus \lien_r$, as well as for their graded components.

Focusing on a fixed summand $\lien_k$, we can try to follow the strategy of \cref{prop:equivariant-quantizations} to recover the graded component $\lien_{k, n}$ of degree $n$ inside $U(\lien)$.
In order to do this, we decompose the tensor product $\lien_{k, 1}^{\otimes n}$ into simple components as an $\liel_k$-module, then select one component isomorphic to $\lien_{k, n}$ and map it to $U(\lien)$ using the multiplication map.
This procedure gives the expected result for $n = 1$ and $n = 2$ but not necessarily for $n \geq 3$, as we explain below.
Note that this is trivial for $n = 1$, while for $n = 2$ it follows from the fact that $\lien_{k, 1} \otimes \lien_{k, 1}$ has a multiplicity-free decomposition, as argued in \cref{cor:uniqueness-quantization}.

First we discuss the case of the component of degree two more explicitly.
As already observed, the simple component $\tilde{\lien}_{k, 2} \subset \lien_{k, 1} \otimes \lien_{k, 1}$ is uniquely determined and antisymmetric.
Fixing a basis $\{ x_i \}_i$ of $\lien_{k, 1}$, we have a basis $\{ \tilde{y}_i \}_i$ of $\tilde{\lien}_{k, 2}$ of the form $\tilde{y}_i = \sum_{a, b} c^i_{a b} x_a \otimes x_b$ with $c^i_{a b} = - c^i_{b a}$.
Its image under the multiplication map $\mu$ is then $y_i = \frac{1}{2} \sum_{i, j} c^i_{a b} [x_a, x_b]$, which makes it clear that we have $\mu(\tilde{\lien}_{k, 2}) = \lien_{k, 2}$ and furthermore $\Delta(y_i) = y_i \otimes 1 + 1 \otimes y_i$.

Now consider the case of the component of degree three.
First we observe that there are at least two components isomorphic to $\lien_{k, 3}$ in the decomposition of $\lien_{k, 1}^{\otimes 3}$. Indeed, denoting them by $C_1$ and $C_2$, they are obtained from the following decompositions
\[
C_1 \subset \tilde{\lien}_{k, 2} \otimes \lien_{k, 1}, \quad
C_2 \subset \lien_{k, 1} \otimes \tilde{\lien}_{k, 2}.
\]
Both of these components are uniquely determined, as the two tensor products have multiplicity-free decompositions.
Then their images in $U(\lien)$ under the multiplication map are candidates for $\lien_{k, 3}$, as in the construction of \cref{prop:equivariant-quantizations}.

However neither of these actually correspond to the component $\lien_{k, 3} \subset \lien$.
To see this, consider an element $\tilde{z} = \sum_{i, j} b_{i j} \tilde{y}_i \otimes x_j$ of $C_1$ and the corresponding element $z = \sum_{i, j} b_{i j} y_i x_j$ in $U(\lien)$.
Then its coproduct is given by $\Delta(z) = \sum_{i, j} b_{i j} \Delta(y_i) \Delta(x_j)$, which is not of the form $z \otimes 1 + 1 \otimes z$. The situation is analogous for the component $C_2$.

On the other hand, $\lien_{k, 3}$ can be recovered from a subspace of $C_1 \oplus C_2$ as follows.
Denoting by $\tau$ the flip map, we first observe that $C_2 = \tau_1 \tau_2 (C_1)$. Then we consider
\[
C = (\id - \tau_1 \tau_2) (C_1) \subset C_1 \oplus C_2.
\]
This means that any element of $C$ has the general form $\tilde{z} = \sum_{i, j} b_{i j} (\tilde{y}_i \otimes x_j - x_j \otimes \tilde{y}_i)$.
Then the corresponding element in $U(\lien)$ is given by $z = \sum_{i, j} b_{i j} [y_i, x_j]$, which immediately shows that $\Delta(z) = z \otimes 1 + 1 \otimes z$, as should be the case for the component $\lien_{k, 3}$.

This illustrates that, in order to recover the classical situation, we should impose extra conditions when quantizing $\lien_{k, n}$ for $n \geq 3$.
As an analogue of the condition $\Delta(z) = z \otimes 1 + 1 \otimes z$ for $z \in \lien_k$, we are going to require that $\bbC \oplus \qnil_k$ should be a left $\Uqb$-coideal.

\section{Adjoint action on Schubert cells}
\label{sec:schubert-cells}

The goal of this section is to connect the equivariant quantizations from the previous section to the quantum Schubert cells $\schub(w)$ introduced by De Concini, Kac and Procesi in \cite{dckp}, corresponding to a Weyl group element $w$, as well as on their twisted versions $\tschub(w_S)$ introduced by Zwicknagl in \cite{zwicknagl}, corresponding to a factorization of the type $w_0 = w_{0, S} w_S$.
The advantage of the latter is that it is invariant under the left adjoint action of the quantized Levi factor, as shown in \cite[Theorem 5.21]{zwicknagl}.

We give a short proof of this result, which tries to elucidate the roles of both $\schub(w_S)$ and $\tschub(w_S)$.
Then we discuss the connection with the equivariant quantizations of $\lien$.

\subsection{Quantum Schubert cells}

For any Weyl group element $w \in W$, De Concini, Kac and Procesi define an algebra $\schub(w) \subseteq \Uqn$ as follows.

\begin{definition}
Fix a reduced decomposition $w = s_{i_1} \cdots s_{i_n}$ and let $E_{\beta_1}, \cdots, E_{\beta_n}$ be the corresponding quantum root vectors.
Then the \emph{quantum Schubert cell} $\schub(w)$ is defined as the subalgebra of $\Uqn$ generated by these quantum root vectors.
\end{definition}

It can be shown that this algebra does not depend on the chosen reduced decomposition of $w$, see the proposition appearing in \cite[Section 2.2]{dckp}.

Now consider a factorization $w_0 = w_{0, S} w_S$ corresponding to a parabolic subalgebra defined by $S \subseteq \simpleroots$. Then in general $\schub(w_S)$ is not invariant under the left adjoint action of $\UqlS$, the quantized Levi factor.
This motivates the following modification of the Schubert cell $\schub(w_S)$ introduced by Zwicknagl, see \cite[Definition 5.4]{zwicknagl}.

\begin{definition}
Consider the factorization $w_0 = w_{0, S} w_S$ with respect to $S \subseteq \simpleroots$.
Then the \emph{twisted quantum Schubert cell} $\tschub(w_S)$ is defined by
\[
\tschub(w_S) := T_{w_{0, S}}(\schub(w_S)).
\]
\end{definition}

Here $T_{w_{0, S}}$ is the Lusztig automorphism corresponding to the longest word $w_{0, S}$ for $S$.
It does not depend on the chosen reduced decomposition for this element.

This definition can be generalized to the relative setting of a Hopf algebra embedding $\Uqk \hookrightarrow \Uqg$, where $\liek$ corresponds to a Dynkin subdiagram of $\lieg$.
In this case the twisted quantum Schubert cell $\tschub(w_S) \subseteq \Uqk$ can be transported to $\Uqg$ in the obvious way.

\subsection{Invariance under the adjoint action}

We are now going to show that the twisted quantum Schubert cell $\tschub(w_S)$ is invariant under the left adjoint action of the quantized Levi factor $\UqlS$, a result originally proven in \cite[Theorem 5.21]{zwicknagl}.

Recall that, given weight vectors $X_\mu$ and $Y_\nu$ of weight $\mu$ and $\nu$ respectively, their $q$-commutator is given by $[X_\mu, Y_\nu]_q = X_\mu Y_\nu - q^{(\mu, \nu)} Y_\nu X_\mu$.
We begin by showing that the two algebras $\schub(w_S)$ and $\tschub(w_S)$ are invariant, in an appropriate sense, under taking $q$-commutators with the generators of the positive part of the quantized Levi factor $\UqlS$.

\begin{lemma}
\label{lem:qcommutator-schubert}
For any $i \in S$ we have:
\begin{enumerate}
\item $[\schub(w_S), E_i]_q \subseteq \schub(w_S)$ for the quantum Schubert cell,
\item $[E_i, \tschub(w_S)]_q \subseteq \tschub(w_S)$ for the twisted quantum Schubert cell.
\end{enumerate}
\end{lemma}

\begin{proof}
(1) By definition the quantum Schubert cell $\schub(w_S)$ is generated by the quantum root vectors $\{ E_{\xi_j} \}_{j = 1}^N$ corresponding to a reduced decomposition of $w_S$.
Consider a factorization $w_0 = w_S w^\prime$ and denote by $\{ E_{\beta_j} \}_{j = 1}^{M + N}$ the corresponding quantum root roots. Then $E_{\beta_j} = E_{\xi_j}$ for $j \in \{ 1, \cdots, N \}$, while the other quantum root vectors belong to the quantized Levi factor $\UqlS$.
It suffices to show that $[E_{\xi_j}, E_i]_q \in \schub(w_S)$ for any $j \in \{ 1, \cdots, N \}$ and $i \in S$, since the general claim follows from the identity
\[
[X_\lambda Y_\mu, Z_\nu]_q = X_\lambda [Y_\mu, Z_\nu]_q + q^{(\mu, \nu)} [X_\lambda, Z_\nu]_q Y_\mu.
\]

We can choose the reduced decomposition for $w^\prime$ in such a way that $\beta_{N + 1} = \alpha_i$, that is $E_{\beta_{N + 1}} = E_i$.
Note that this does not change the quantum quantum root vectors $\{ E_{\xi_j} \}_{j = 1}^N$.
Then for any $j \in \{ 1, \cdots, N \}$ we consider the $q$-commutator $[E_{\xi_j}, E_i]_q = [E_{\beta_j}, E_{\beta_{N + 1}}]_q$.
By the $q$-commutation relations from \cref{prop:q-commutation-relations} we get $[E_{\beta_j}, E_{\beta_{N + 1}}]_q = \sum_\bfM c_\bfM E^\bfM$, where the sum can be restricted to the tuples $(m_1, \cdots, m_{M + N})$ with $m_k = 0$ for $k \leq j$ and $k \geq N + 1$.
Since $E_{\beta_k} \in \schub(w_S)$ for $k \in \{ j + 1, \cdots, N \}$, we get the result.

(2) The proof is similar to (1). Choosing a reduced decomposition for $w_0 = w_{0, S} w_S$, we have the quantum root vectors $\{ E'_{\gamma_j} \}_{j = 1}^{M + N}$.
Then the generators $\{ E'_{\xi_j} \}_{j = 1}^N$ of $\tschub(w_S)$ are given by $E'_{\xi_j} = E'_{\gamma_{M + j}}$ for $j \in \{ 1, \cdots, N \}$, while the other quantum root vectors belong to the quantized Levi factor $\UqlS$.
It suffices to prove the claim for the generators, since
\[
[X_\lambda, Y_\mu Z_\nu]_q = [X_\lambda, Y_\mu]_q Z_\nu + q^{(\lambda, \mu)} Y_\mu [X_\lambda, Z_\nu]_q.
\]

As above, we can choose the reduced decomposition for $w_{0, S}$ in such a way that $E'_{\gamma_M} = E_i$.
Then we consider the $q$-commutator $[E_i, E'_{\xi_j}]_q = [E'_{\gamma_M}, E'_{\gamma_{M + j}}]_q$.
An argument using the $q$-commutation relations as in (1) gives the result.
\end{proof}

The previous result shows that there is not much difference between $\schub(w_S)$ and $\tschub(w_S)$ from the point of view of the $q$-commutation relations.
However an important difference arises when we consider the left adjoint action, since we have the identity $E_i \triangleright X_\mu = [E_i, X_\mu]_q$ (in our conventions).
Then the previous $q$-commutation relations can be used to prove the following result originally due to Zwicknagl, see \cite[Theorem 5.21(b)]{zwicknagl}.

\begin{proposition}
\label{prop:adjoint-action-schubert}
For any subset $S \subseteq \simpleroots$, the twisted quantum Schubert cell $\tschub(w_S)$ is invariant under the left adjoint action of the quantized Levi factor $\UqlS$.
\end{proposition}

\begin{proof}
Recall that the quantized Levi factor $\UqlS$ is generated by the Cartan elements, together with $E_i$ and $F_i$ with $i \in S$.
It is clear that $\tschub(w_S)$ is invariant under the action of the Cartan elements, since it is generated by weight vectors.
Next we note that, in our conventions for the coproduct, we have the identities
\[
\begin{split}
E_i \triangleright X_\mu & = E_i X_\mu - K_i X_\mu K_i^{-1} E_i = [E_i, X_\mu]_q, \\
F_i \triangleright X_\mu & = F_i X_\mu K_i - X_\mu F_i K_i = [F_i, X_\mu] K_i.
\end{split}
\]
Here $X_\mu$ is any weight vector.
We note that in the first case we have a $q$-commutator, while in the second case we have a regular commutator.
Then the first identity and \cref{lem:qcommutator-schubert} imply that $E_i \triangleright \tschub(w_S) \subseteq \tschub(w_S)$ for any $i \in S$.

Now consider $F_i \triangleright X_\mu = [F_i, X_\mu] K_i$ for $i \in S$.
Using the automorphism $T_{w_{0, S}}^{-1}$ we get
\[
T_{w_{0, S}}^{-1}(F_i \triangleright X_\mu) = [T_{w_{0, S}}^{-1}(F_i), T_{w_{0, S}}^{-1}(X_\mu)] T_{w_{0, S}}^{-1}(K_i).
\]
We must have $w_{0, S}(\alpha_i) = - \alpha_j$ for some $j \in S$, since $w_{0, S}$ is the longest word corresponding to $S$.
We rewrite this as $w_{0, S}(\alpha_j) = - \alpha_i$, since $w_{0, S}$ is involutive.
We can factorize $w_{0, S} = s_i w$ with $\ell(w) = \ell(w_{0, S}) - 1$ and $w(\alpha_j) = \alpha_i$.
Then it follows from \cite[Proposition 8.20]{jantzen} that $T_w(E_j) = E_i$.
Using this and \eqref{eq:lusztig-relations} we compute
\[
T_{w_{0, S}} (E_j K_j) = T_i T_w (E_j K_j) = T_i (E_i K_i) = - F_i.
\]
From this it follows that $T_{w_{0, S}}^{-1}(F_i) = - E_j K_j$.
This lets us rewrite
\[
T_{w_{0, S}}^{-1}(F_i \triangleright X_\mu) = - [E_j K_j, T_{w_{0, S}}^{-1}(X_\mu)] K_j^{-1}.
\]
Since $X_\mu$ has weight $\mu$, the element $T_{w_{0, S}}^{-1}(X_\mu)$ has weight $w_{0, S}^{-1}(\mu) = w_{0, S}(\mu)$. Then
\[
\begin{split}
T_{w_{0, S}}^{-1}(F_i \triangleright X_\mu)
& = T_{w_{0, S}}^{-1}(X_\mu) E_j - q^{(\alpha_j, w_{0, S}(\mu))} E_j T_{w_{0, S}}^{-1}(X_\mu) \\
& = [T_{w_{0, S}}^{-1}(X_\mu), E_j]_q.
\end{split}
\]
Therefore $F_i \triangleright X_\mu = T_{w_{0, S}} \left( [T_{w_{0, S}}^{-1}(X_\mu), E_j]_q \right)$, where $w_{0, S}(\alpha_i) = - \alpha_j$ and $j \in S$.

Applying this identity to any element of $\tschub(w_S)$ we get
\[
F_i \triangleright \tschub(w_S) \subseteq 
T_{w_{0, S}} \left( [T_{w_{0, S}}^{-1}(\tschub(w_S)), E_j]_q \right) = T_{w_{0, S}} \left( [\schub(w_S), E_j]_q \right).
\]
By \cref{lem:qcommutator-schubert} we have $[U(w_S), E_j]_q \subseteq U(w_S)$ for any $j \in S$. Hence
\[
F_i \triangleright \tschub(w_S) \subseteq T_{w_{0, S}} (\schub(w_S)) = \tschub(w_S),
\]
which concludes the proof.
\end{proof}

\begin{remark}
The result remains valid if we replace the Lusztig automorphisms $T_i$ with another choice $\tilde{T}_i$ such that $\tilde{T}_i(X_\mu) = t_{i, \mu} T_i(X_\mu)$ on weight vectors, as in \cref{rem:lusztig-automorphisms}.
\end{remark}

\subsection{Comparison}

Now we discuss the connection between the twisted quantum Schubert cells and the equivariant quantizations $\qnil = \qnil_1 \oplus \cdots \oplus \qnil_r$ as in \cref{prop:equivariant-quantizations}.

First we observe that, for any increasing sequence $S_1 \subset \cdots \subset S_r$ of proper subsets of $\simpleroots$, we have a corresponding factorization of the longest word of $\lieg$ of the form
\[
w_0 = w_1 \cdots w_r, \quad
\ell(w_0) = \ell(w_1) + \cdots + \ell(w_r).
\]
To obtain it, start from the factorization $w_0 = w_{0, S_r} w_{S_r}$ with respect to $S_r \subset \simpleroots$ and set $w_r = w_{S_r}$.
Then $w_{0, S_r}$ is the longest word of the semisimple Lie algebra $\lieg_r \subset \lieg$.
Proceeding in this way, we factorize $w_{0, S_{k + 1}} = w_{0, S_k} w_{S_k}$ with respect to $S_k \subset S_{k + 1}$ and set $w_k = w_{S_k}$.
This gives the factorization $w_0 = w_1 \cdots w_r$ with respect to the sequence $S_1 \subset \cdots \subset S_r$, where in the last step we use $w_{0, \emptyset} = 1$.
We note that a factorization of this type for $w_0$ is related to the notion of \emph{nice decomposition} introduced by Littelmann, see \cite[Section 4]{littelmann}.

Note that the Weyl group element $w_k$ occurs in the factorization of the longest word of $\lieg_{k + 1}$ with respect to the inclusion $S_k \subset S_{k + 1}$.
We have an embedding $U_q(\lieg_{k + 1}) \hookrightarrow \Uqg$ of $\Uqlk$-modules obtained by mapping the generators in the obvious way.
Then the twisted quantum Schubert cell $\tschub(w_k) \subset U_q(\lieg_{k + 1})$ can be transported to $\Uqg$ via this embedding, and we denote it by the same symbol in the next result.

\begin{proposition}
\label{prop:schubert-component}
Let $\qnil = \qnil_1 \oplus \cdots \oplus \qnil_r$ be an equivariant quantization corresponding to $S_1 \subset \cdots \subset S_r$. Then we have $\qnil_k \subset \tschub(w_k)$ for any $k \in \{ 1, \cdots, r \}$.
Moreover $\qnil_{k, 1}$ coincides with the degree-one component of $\tschub(w_k)$ with respect to the grading by $\alpha_{i_k}$.
\end{proposition}

\begin{proof}
For any $k \in \{ 1, \cdots, r \}$ we have $E_{i_k} \in \tschub(w_k)$.
Indeed, for the enumeration of the positive roots corresponding to $w_0 = w_{0, S} w_S$, those of the Levi factor $\liel_S$ occur before those of the positive nilradical $\lien_S$.
Recall that the twisted quantum Schubert cell $\tschub(w_k)$ is invariant under the left adjoint action of $\Uqlk$, by \cref{prop:adjoint-action-schubert}.
Then this action on $E_{i_k}$ generates the degree-one component of $\tschub(w_k)$ with respect to $\alpha_{i_k}$, since this is not a root of $\liel_k$.
By \cref{lem:action-degree-one} the degree-one component $\qnil_{k, 1}$ is also generated by $E_{i_k}$ under the left adjoint action of $\Uqlk$, hence the two must coincide.
Finally, by construction any graded component $\qnil_{k, n}$ is contained in the $n$-fold product of $\qnil_{k, 1}$, hence it is contained in $\tschub(w_k)$.
\end{proof}

This result implies that each degree-one component $\qnil_{k, 1}$ consists of quantum root vectors corresponding to some reduced decomposition of $w_0$, which makes its description more explicit.
We are going to make use of this fact later in \cref{prop:comparison-heckenberger-kolb}.

\section{Quasitriangular Hopf algebras and transmutation}
\label{sec:quasitriangular}

In this section we momentarily abandon the quantized enveloping algebra $\Uqg$ to discuss some general results on quasitriangular Hopf algebras, which are going to be used in the next section.
The overall theme is the comparison of two left actions of a Hopf algebra $H$ on the tensor product $H \otimes H$.
As we are going to discuss, in the quasitriangular case this comparison can be achieved by a form of transmutation, a notion introduced by Majid.
We provide some details for this result, as we could not find it in the literature in the form we require.
Finally we discuss a formula, which we refer to as the quantum shuffle product, for a certain type of products in the braided tensor product algebra $H \otimes_\sigma H$.

\subsection{Actions on the tensor product}

Let $H$ be a Hopf algebra and denote as before the left adjoint action of $H$ on itself by $x \triangleright y = x_{(1)} y S(x_{(2)})$.
Then we have the tensor product action of $H$ on $H \otimes H$, denoted by the same symbol, given by
\[
x \triangleright (y \otimes z) := x_{(1)} \triangleright y \otimes x_{(2)} \triangleright z.
\]
In general we have $x \triangleright (y y' \otimes z z') \neq x_{(1)} \triangleright (y \otimes z) \cdot x_{(2)} \triangleright (y' \otimes z')$, so that $H \otimes H$ is not a $H$-module algebra with respect to $\triangleright$ (but see \cref{prop:braided-product} for the quasitriangular case).
Furthermore let us consider the coproduct $\Delta(x \triangleright y)$. Then a quick computation shows that
\[
\Delta(x \triangleright y) = x_{(1)} y_{(1)} S(x_{(3)}) \otimes x_{(2)} \triangleright y_{(2)}.
\]
In general this does not coincide with $x \triangleright \Delta(y)$, unless $H$ is cocommutative.

On the other hand, we can introduce a second left action $\blacktriangleright$ of $H$ on $H \otimes H$, which reduces to the previous identity on the subspace $\Delta(H) \subseteq H \otimes H$, given by
\[
x \blacktriangleright (y \otimes z) := x_{(1)} y S(x_{(3)}) \otimes x_{(2)} \triangleright z.
\]
Below we discuss some of its basic properties.

\begin{proposition}
The map $\blacktriangleright$ satisfies the following properties:
\begin{enumerate}
\item It is a left action of $H$ on $H \otimes H$.
\item It makes $H \otimes H$ into a $H$-module algebra.
\item It satisfies $x \blacktriangleright \Delta(y) = \Delta(x \triangleright y)$.
\end{enumerate}
\end{proposition}

\begin{proof}
(1) Using the fact that $\triangleright$ is a left action and $S$ an anti-homomorphism, we compute
\[
\begin{split}
x x' \blacktriangleright (y \otimes z)
& = x_{(1)} x'_{(1)} y S(x_{(3)} x'_{(3)}) \otimes x_{(2)} x'_{(2)} \triangleright z \\
& = x_{(1)} x'_{(1)} y S(x'_{(3)}) S(x_{(3)}) \otimes x_{(2)} \triangleright x'_{(2)} \triangleright z \\
& = x \blacktriangleright x' \blacktriangleright (y \otimes z).
\end{split}
\]
(2) Using the module algebra property of $\triangleright$ and $S(x_{(1)}) x_{(2)} = \varepsilon(x)$ we compute
\[
\begin{split}
x \blacktriangleright (y y' \otimes z z')
& = x_{(1)} y y' S(x_{(4)}) \otimes (x_{(2)} \triangleright z) (x_{(3)} \triangleright z') \\
& = x_{(1)} y S(x_{(3)}) x_{(4)} y' S(x_{(6)}) \otimes (x_{(2)} \triangleright z) (x_{(5)} \triangleright z') \\
& = x_{(1)} \blacktriangleright (y \otimes z) \cdot x_{(2)} \blacktriangleright (y' \otimes z').
\end{split}
\]
(3) This is immediate from the definition of $\blacktriangleright$.
\end{proof}

We are interested in comparing the two actions $\triangleright$ and $\blacktriangleright$, in some suitable sense.
This can be done in the case of a \emph{quasitriangular} Hopf algebra, which we now review.

\subsection{Quasitriangular Hopf algebras}

Let $H$ be a quasitriangular Hopf algebra.
This means that $H$ admits an \emph{R-matrix}, which is an invertible element $\calR \in H \otimes H$ such that
\[
\begin{gathered}
\calR \cdot \Delta(h) \cdot \calR^{-1} = \Delta^\op(h), \quad h \in H, \\
(\Delta \otimes \id) (\calR) = \calR_{1 3} \cdot \calR_{2 3}, \quad
(\id \otimes \Delta) (\calR) = \calR_{1 3} \cdot \calR_{1 2}.
\end{gathered}
\]
From these properties it follows that
\[
\begin{gathered}
(\varepsilon \otimes \id) (\calR) = 1, \quad
(\id \otimes \varepsilon) (\calR) = 1, \\
(S \otimes \id) (\calR) = \calR^{-1}, \quad
(\id \otimes S) (\calR^{-1}) = \calR.
\end{gathered}
\]

Given an R-matrix $\calR$ as above, the corresponding \emph{braiding} is defined to be $\hat{R} := \tau \circ \calR$, where $\tau$ denotes the flip map.
The name is related to the fact that $\hat{R}$ turns the category of finite-dimensional $H$-modules into a braided monoidal category.
We also denote this braiding by $\sigma$ in the following, mostly for typographical reasons.

The following results is well-known, see for instance \cite[Corollary 9.2.13]{majid-book} (and note that the result in this reference should read $H$-modules as opposed to $H$-comodules).

\begin{proposition}
\label{prop:braided-product}
Let $H$ be a quasitriangular Hopf algebra with braiding $\sigma$.
Let $A$ and $B$ be left $H$-module algebras.
Then the \emph{braided tensor product} $A \otimes_\sigma B$ is the algebra given by $A \otimes B$ as a vector space and with product
\[
(a \otimes b) \cdot_\sigma (c \otimes d) = \mu_1 \mu_3 \sigma_2 (a \otimes b \otimes c \otimes d).
\]
It is a left $H$-module algebra with respect to the tensor product action $\triangleright$.
\end{proposition}

In this result $\mu: A \otimes A \to A$ denotes the multiplication map on an algebra $A$ and we use the standard leg-numbering notation, a convention we shall keep in the following.

\subsection{Transmutation}

As discussed at the beginning of this section, the two left actions $\triangleright$ and $\blacktriangleright$ of $H$ on $H \otimes H$ do not coincide, in general.
To compare them we can use the theory of \emph{transmutation}, due to Majid \cite{majid-transmutation, majid-braided}.
This takes as input a quasitriangular Hopf algebra $H$ and produces a braided Hopf algebra $\underline{H}$.
In particular we can compare the subspaces $\Delta(H)$ and $\underline{\Delta}(H)$ of $H \otimes H$, where $\underline{\Delta}$ is the a modified coproduct of $\underline{H}$.
For more details on this see for instance \cite[Theorem 7.4.2 and Example 9.4.9]{majid-book}.

The setting we consider here is a bit different, since we want to work directly with $H \otimes H$ rather than the subspaces described above.
As we could not find this result stated in this precise form anywhere, we provide some details of its proof.

\begin{theorem}
\label{thm:transmutation}
Let $H$ be a quasitriangular Hopf algebra with universal R-matrix $\calR = \sum_i a_i \otimes b_i$ and braiding $\sigma$.
Define the map $\tmap: H \otimes H \to H \otimes H$ by
\[
\tmap(x \otimes y) = \sum_i x b_i \otimes a_i \triangleright y.
\]
Then we have an isomorphism of $H$-module algebras $\tmap: (H \otimes_\sigma H, \triangleright) \to (H \otimes H, \blacktriangleright)$.
\end{theorem}

\begin{proof}
\textbf{Invertibility}. Using the properties of the R-matrix one easily checks that
\[
\tmap^{-1}(x \otimes y) = \sum_i x S(b_i) \otimes a_i \triangleright y.
\]
\textbf{Equivariance}. Using the definition of the left action $\blacktriangleright$ we compute
\[
\begin{split}
z \blacktriangleright \tmap(x \otimes y)
& = \sum_i z_{(1)} x b_i S(z_{(3)}) \otimes z_{(2)} a_i \triangleright y \\
& = \sum_i z_{(1)} x b_i S(z_{(3)}) \otimes z_{(2)} a_i S(z_{(4)}) z_{(5)} \triangleright y.
\end{split}
\]
Using $\calR_{2 1} \cdot \Delta^\op(x) = \Delta(x) \cdot \calR_{2 1}$ and $\Delta^\op(S(x)) = S(x_{(1)}) \otimes S(x_{(2)})$ we get
\[
\begin{split}
z \blacktriangleright \tmap(x \otimes y)
& = \sum_i z_{(1)} x S(z_{(4)}) b_i \otimes z_{(2)} S(z_{(3)}) a_i z_{(5)} \triangleright y \\
& = \sum_i z_{(1)} x S(z_{(2)}) b_i \otimes a_i z_{(3)} \triangleright y \\
& = \sum_i (z_{(1)} \triangleright x) b_i \otimes a_i \triangleright z_{(2)} \triangleright y = \tmap(z \triangleright (x \otimes y)).
\end{split}
\]
\textbf{Homomorphism}. The coproduct identities for $\calR = \sum_i a_i \otimes b_i$ can be written as
\begin{equation}
\label{eq:identities-R}
\begin{split}
\sum_i a_{i (1)} \otimes a_{i (2)} \otimes b_i & = \sum_{i, j} a_i \otimes a_j \otimes b_i b_j, \\
\sum_i a_i \otimes b_{i (1)} \otimes b_{i (2)} & = \sum_{i, j} a_i a_j \otimes b_j \otimes b_i.
\end{split}
\end{equation}
Using $x y = (x_{(1)} \triangleright y) x_{(2)}$ and the second identity in \eqref{eq:identities-R}, we compute
\[
\begin{split}
\tmap(x \otimes y) \cdot \tmap(x' \otimes y')
& = \sum_{i, k} x b_i x' b_k \otimes (a_i \triangleright y) (a_k \triangleright y') \\
& = \sum_{i, k} x (b_{i (1)} \triangleright x') b_{i (2)} b_k \otimes (a_i \triangleright y) (a_k \triangleright y') \\
& = \sum_{i, j, k} x (b_j \triangleright x') b_i b_k \otimes (a_i \triangleright a_j \triangleright y) (a_k \triangleright y').
\end{split}
\]
On the other hand, using the first identity in \eqref{eq:identities-R} we compute
\[
\begin{split}
\tmap(x \tilde{x} \otimes \tilde{y} y')
& = \sum_i x \tilde{x} b_i \otimes (a_{i (1)} \triangleright \tilde{y}) (a_{i (2)} \triangleright y') \\
& = \sum_{i, k} x \tilde{x} b_i b_k \otimes (a_i \triangleright \tilde{y}) (a_k \triangleright y').
\end{split}
\]
Setting $\tilde{x} \otimes \tilde{y} = \sum_j b_j \triangleright x' \otimes a_j \triangleright y$ gives $\tmap(x \tilde{x} \otimes \tilde{y} y') = \tmap(x \otimes y) \cdot \tmap(x' \otimes y')$.
Finally we note that since $\tilde{x} \otimes \tilde{y} = \sigma(x' \otimes y)$ we have $x \tilde{x} \otimes \tilde{y} y' = (x \otimes y) \cdot_\sigma (x' \otimes y')$.
\end{proof}

\begin{remark}
The relation to Majid's transmutation is described as follows.
If we restrict the inverse map $\tmap^{-1}$ to the subspace $\Delta(H) \subseteq H \otimes H$, we obtain
\[
\tmap^{-1} \circ \Delta(h) = \sum_i h_{(1)} S(b_i) \otimes a_i \triangleright h_{(2)}.
\]
Comparing with the braided coproduct $\underline{\Delta}$ of $H$, as defined  in \cite[Theorem 7.4.2]{majid-book}, we see that we have the relation $\tmap^{-1} \circ \Delta = \underline{\Delta}$ between the two coproducts.
\end{remark}

The previous result can be generalized to the relative setting, where $K$ is a quasitriangular Hopf algebra and we have an injective Hopf algebra map $\iota: K \to H$.
Since we need this setting in the next section, we spell out some of the relevant details.

We can consider $K$ as a Hopf subalgebra of $H$ via the map $\iota$. As such, we get a left adjoint action of $K$ on $H$ as the restriction of the left adjoint action of $H$ on itself. Similarly, we obtain the left actions $\triangleright$ and $\blacktriangleright$ of $K$ on the tensor product $H \otimes H$.
Finally, transporting the R-matrix $\calR_K$ of $K$ via the map $\iota$, we obtain the following.

\begin{corollary}
\label{cor:relative-transmutation}
Let $K$ be a quasitriangular Hopf algebra with universal R-matrix $\calR_K$ and braiding $\sigma_K$.
Let $H$ be a Hopf algebra and $\iota: K \to H$ and injective Hopf algebra map.
Write $(\iota \otimes \iota) (\calR_K) = \sum_i a_i \otimes b_i$. Define the map $\tmap_K: H \otimes H \to H \otimes H$ by
\[
\tmap_K(x \otimes y) = \sum_i x b_i \otimes a_i \triangleright y.
\]
Then we have an isomorphism of $K$-module algebras $\tmap_K: (H \otimes_{\sigma_K} H, \triangleright) \to (H \otimes H, \blacktriangleright)$.
\end{corollary}

\begin{proof}
Write $\calR = (\iota \otimes \iota) (\calR_K) \in H \otimes H$. Since $\iota$ is an injective Hopf algebra map we have that $\calR$ is invertible and satisfies the properties
\[
\begin{gathered}
(\Delta \otimes \id) (\calR) = \calR_{1 3} \cdot \calR_{2 3}, \quad
(\id \otimes \Delta) (\calR) = \calR_{1 3} \cdot \calR_{1 2}.
\end{gathered}
\]
On the other hand we have $\calR \cdot \Delta(k) \cdot \calR^{-1} = \Delta^\op(k)$ for any $k \in K$, where here we consider $K \subset H$ via the inclusion $\iota$.
Then it is easy to check that the proof of \cref{thm:transmutation} remains valid in this setting, with the obvious exception of $K$-equivariance replacing $H$-equivariance.
\end{proof}

\subsection{Quantum shuffle product}

We conclude this section by discussing an explicit formula for braided products in $H \otimes_\sigma H$ of the form
\[
\pmap^{(n)} (h_1, \cdots, h_n) = (h_1 \otimes 1 + 1 \otimes h_1) \cdot_\sigma \cdots \cdot_\sigma (h_n \otimes 1 + 1 \otimes h_n).
\]
We should think of $\pmap^{(n)} (h_1, \cdots, h_n)$ as the braided product of "primitive" elements in $H \otimes H$ (more precisely, primitive with respect to the modified coproduct $\underline{\Delta}$).

We refer to the formula obtained below as the \emph{quantum shuffle product}, for reasons soon to be apparent.
The result is known, see for instance \cite[Theorem 6.4.9]{hs-book}, where it appears in the form of the comultiplication for a braided tensor algebra.
It is also essentially dual to the construction of the quantum shuffle Hopf algebra defined by Rosso in \cite{rosso}.
Nevertheless, we find it useful to give a short proof of this result below.

The set of \emph{$(i, j)$-shuffles} with $n = i + j$ consists of permutations $s \in S_n$ such that
\[
s(1) < \cdots < s(i), \quad
s(i + 1) < \cdots < s(n).
\]
We write $S_{i, j}$ for the set of $(i, j)$-shuffles and note that it has cardinality $\binom{n}{i}$.
Below we actually consider the inverse shuffles, which are sometimes called \emph{unshuffles}.
Given any reduced decomposition $s = s_{i_1} \cdots s_{i_m}$ of $s \in S_n$, we define
\[
\sigma_s := \sigma_{i_1} \cdots \sigma_{i_m},
\]
where we use the leg-numbering notation as usual.
This does not depend on the chosen reduced decomposition, since $\sigma$ satisfies the braid equation.
Furthermore, consider the multiplication maps $\mu_{i, j} : H^{\otimes n} \to H^i \otimes H^j$ for $n = i + j$, given by
\[
\mu_{i, j} (h_1 \otimes \cdots \otimes h_n) = h_1 \cdots h_i \otimes h_{i + 1} \cdots h_n.
\]
We note the special cases $\mu_{0, n} (h_1 \otimes \cdots \otimes h_n) = 1 \otimes h_1 \cdots h_n$ and $\mu_{n, 0} (h_1 \otimes \cdots \otimes h_n) = h_1 \cdots h_n \otimes 1$.
Using this notation, we obtain the promised quantum shuffle product as follows.

\begin{proposition}
\label{prop:shuffle-product}
For any $n \geq 1$ we have the formula
\[
\pmap^{(n)} (h_1, \cdots, h_n) = \sum_{i + j = n} \mu_{i, j} \sum_{s^{-1} \in S_{i, j}} \sigma_s (h_1 \otimes \cdots \otimes h_n).
\]
\end{proposition}

\begin{proof}
First note that we can write
\[
\pmap^{(n + 1)} (h_1, \cdots, h_{n + 1}) = \pmap^{(n)}(h_1, \cdots, h_n) \cdot_\sigma (h_{n + 1} \otimes 1 + 1 \otimes h_{n + 1}).
\]
Recall that $(a \otimes b) \cdot_\sigma (c \otimes d) = \mu_1 \mu_3 \sigma_2 (a \otimes b \otimes c \otimes d)$ by \cref{prop:braided-product}.
Using the compatibility between the multiplication, the braiding and the unit, we obtain
\[
(a \otimes b) \cdot_\sigma (c \otimes 1) = \mu_1 \sigma_2 (a \otimes b \otimes c), \quad
(a \otimes b) \cdot_\sigma (1 \otimes d) = \mu_2 (a \otimes b \otimes d).
\]
This means that $\pmap^{(n + 1)} = (\mu_1 \sigma_2 + \mu_2)(\pmap^{(n)} \otimes \id)$ as maps on $H^{\times (n + 1)}$.

The claimed formula holds for $n = 1$, since $\pmap^{(1)}(h_1) = (\mu_{1, 0} + \mu_{0, 1})(h_1)$, so suppose the it holds for some generic $n$.
Let us note that $\mu_{i, j} = \mu_1^{i - 1} \mu_{i + 1}^{j - 1}$, where this expression must be interpreted appropriately for the boundary cases $i = 0$ and $j = 0$.
Let us also abbreviate $B_{i, j} = \sum_{s^{-1} \in S_{i, j}} \sigma_s$.
Then we are assuming that
\[
\pmap^{(n)} = \sum_{i + j = n} \mu_1^{i - 1} \mu_{i + 1}^{j - 1} B_{i, j}.
\]
Using this assumption we compute
\[
\begin{split}
\pmap^{(n + 1)}
& = (\mu_1 \sigma_2 + \mu_2)(\pmap^{(n)} \otimes \id) \\
& = \sum_{i + j = n} \mu_1^i \sigma_{i + 1} \mu_{i + 1}^{j - 1} (B_{i, j} \otimes \id) + \sum_{i + j = n} \mu_1^{i - 1} \mu_{i + 1}^j (B_{i, j} \otimes \id) \\
& = \sum_{i + j = n + 1} \left( \mu_1^{i - 1} \sigma_i \mu_i^{j - 1} (B_{i - 1, j} \otimes \id) + \mu_1^{i - 1} \mu_{i + 1}^{j - 1} (B_{i, j - 1} \otimes \id) \right).
\end{split}
\]
By induction one shows that $\sigma_i \mu_i^m = \mu_{i + 1}^m \sigma_i \cdots \sigma_{i + m}$ for any $m$.
This lets us rewrite
\[
\pmap^{(n + 1)} = \sum_{i + j = n + 1} \mu_1^{i - 1} \mu_{i + 1}^{j - 1} \left( \sigma_i \cdots \sigma_n (B_{i - 1, j} \otimes \id) + (B_{i, j - 1} \otimes \id) \right).
\]
Then, as in the proof of \cite[Proposition 6]{rosso}, we can use the identity
\[
B_{i, j} = \sigma_i \cdots \sigma_n (B_{i - 1, j} \otimes \id) + (B_{i, j - 1} \otimes \id).
\]
This is obtained by observing that, for any shuffle $w \in S_{i, j}$ with $i + j = n + 1$, we have either $w(i) = n + 1$ or $w(n + 1) = n + 1$.
Using this identity gives the result.
\end{proof}

\section{Coideal property}
\label{sec:coideal}

We now come back to the equivariant quantizations $\qnil = \qnil_1 \oplus \cdots \oplus \qnil_r$ associated to increasing sequences $S_1 \subset \cdots \subset S_r$ of proper subsets of $\simpleroots$.
These were shown to exist in \cref{prop:equivariant-quantizations}, but are not guaranteed to be unique for a given sequence, since the construction depends on decomposing appropriate tensor products, where the simple components can appear with multiplicities larger than one.
Our main goal in this section is to study whether $\bbC \oplus \qnil$ is a left coideal, which as a byproduct allows us to get uniqueness in some cases.

The main tool for this analysis is the map $\tmap$ introduced in the previous section, or rather its counterpart $\tmap_S$ corresponding to a subset $S \subseteq \simpleroots$, as we explain below.
Together with the quantum shuffle product formula, this allows us to verify the coideal condition, with the possible exception of some cases corresponding to exceptional Lie algebras.

\subsection{General considerations}

Let us consider an arbitrary subset $S \subseteq \simpleroots$ of the simple roots and the corresponding quantized Levi factor $\UqlS$.
Then both $\UqlS$ and $\Uqg$ are quasitriangular Hopf algebras and we can apply the results of \cref{sec:quasitriangular}.

To be precise, in this case the universal R-matrix does not belong to the algebraic tensor product, but rather a completion of it. We do not dwell on this point, which only requires minor modifications and is well-understood in the literature.
Moreover, we mostly work in the situation where one of the leg of the tensor product is in the locally-finite part, in which case we have no issues regarding infinite sums for the R-matrix.

\begin{proposition}
Given any $S \subseteq \simpleroots$, denote by $\calR_S = \sum_i a_i \otimes b_i$ the universal R-matrix corresponding to $\UqlS \subseteq \Uqg$ and by $\sigma_S$ the corresponding braiding.
Define the map
\[
\tmapS: \Uqg \otimes \Uqg \to \Uqg \otimes \Uqg, \quad
\tmapS(x \otimes y) = \sum_i x b_i \otimes a_i \triangleright y.
\]
Then we have an isomorphism of $\UqlS$-module algebras
\[
\tmapS: (\Uqg \otimes_{\sigma_S} \Uqg, \triangleright) \to (\Uqg \otimes \Uqg, \blacktriangleright).
\]
\end{proposition}

\begin{proof}
This is simply \cref{cor:relative-transmutation} applied to the case $\UqlS \subseteq \Uqg$.
\end{proof}

Observe that in this result we have the braiding $\sigma_S = \tau \circ \calR_S$ corresponding to $\UqlS$.
For the rest of this section we use the shorter notation $\sigma = \sigma_S$, for fixed $S$.

Also note that the previous result is valid for any choice of R-matrix $\calR_S$.
In the following we make a precise choice, which we write schematically as
\[
\calR_S = \kappa \circ \calQ_S.
\]
Here $\kappa(x_\mu \otimes y_\nu) := q^{(\mu, \nu)} x_\mu \otimes y_\nu$ for weight vectors, while $\calQ_S$ is known as the \emph{quasi R-matrix}.
The only properties we require of $\calQ_S$ is that it can be written as
\[
\calQ_S = \sum_{\beta \geq 0} \calQ_{S, \beta}, \quad
\calQ_{S, \beta} \in \UqlS_{- \beta} \otimes \UqlS_\beta,
\]
and moreover $\calQ_{S, 0} = 1 \otimes 1$.
See for instance the discussion in \cite[Section 7]{jantzen}, keeping in mind the general result that if $\calR$ is a universal R-matrix then so is $\calR_{2 1}^{-1}$.
We also observe that for the corresponding braiding $\sigma = \tau \circ \calR_S$ we have
\[
\sigma(x_\mu \otimes y_\nu) = q^{(\mu, \nu)} y_\nu \otimes x_\mu + \sum_i y_i \otimes x_i,
\]
where each $x_i$ has weight lower than $x_\mu$ and each $y_i$ has weight higher than $y_\nu$.

With the R-matrix $\calR_S$ as above, we have the following result.

\begin{proposition}
\label{prop:properties-T}
Let $V \subset \Uqg$ be invariant under the left adjoint action of $\UqlS$.

\begin{enumerate}
\item The map $\tmapS$ preserves the subspace $\Uqb \otimes V$.
\item If $X_\mu \in V$ is a lowest weight vector of weight $\mu$, then $\tmapS(1 \otimes X_\mu) = K_\mu \otimes X_\mu$.
\end{enumerate}
\end{proposition}

\begin{proof}
(1) This follows from the formula $\tmapS(x \otimes y) = \sum_i x b_i \otimes a_i \triangleright y$ defining the map $\tmapS$, since we have $\calQ_S = \sum_{\beta \geq 0} \calQ_{S, \beta}$ and $\calQ_{S, \beta} \in \UqlS_{- \beta} \otimes \UqlS_\beta$.

(2) Consider weight vectors $X_\mu$ and $Y_\nu$, where $X_\mu$ is a lowest weight vector with respect to $\triangleright$ as above.
Recalling the given characterization of $\calQ_S$, we have
\[
\calQ_S \triangleright (X_\mu \otimes Y_\nu) = \calQ_{S, 0} \triangleright (X_\mu \otimes Y_\nu) = X_\mu \otimes Y_\nu.
\]
Then for the universal R-matrix we obtain
\[
\calR_S \triangleright (X_\mu \otimes Y_\nu) = \kappa (X_\mu \otimes Y_\nu) = q^{(\mu, \nu)} X_\mu \otimes Y_\nu = X_\mu \otimes K_\mu \triangleright Y_\nu.
\]
From this identity it follows that $\tmapS(1 \otimes X_\mu) = K_\mu \otimes X_\mu$.
\end{proof}

\begin{remark}
We remark that for the universal R-matrix $\calR_S$ we use $\kappa(x_\mu \otimes y_\nu) = q^{(\mu, \nu)} x_\mu \otimes y_\nu$, where $(\cdot, \cdot)$ is the inner product corresponding to the weight lattice of $\lieg$ (instead of the corresponding inner product for the semisimple part of $\liel_S$).
This is possible since the notions of weights for $\lieg$ and $\liel_S$ coincide, as the latter contains the Cartan subalgebra $\lieh$.
This choice guarantees that $\kappa (X_\mu \otimes Y_\nu) = X_\mu \otimes K_\mu \triangleright Y_\nu$, which we used in the proof above.

We also observe that one can realize $\calR_S$ by first considering the universal R-matrix of $\Uqg$ of a similar form and then projecting it down to the subspace $\UqlS \otimes \UqlS$.
\end{remark}

The previous result implies the following properties.

\begin{corollary}
\label{cor:action-T-simple}
Let $V \subset \Uqg$ be a simple $\UqlS$-module under the left adjoint action. Then
\[
\tmapS(V \otimes 1) = V \otimes 1, \quad
\tmapS(1 \otimes V) \subset \UqlS^{\geq 0} \otimes V,
\]
where $\UqlS^{\geq 0} = \UqlS \cap \Uqb$ is the non-negative part of $\UqlS$.
\end{corollary}

\begin{proof}
The first property is obvious, since $\tmapS(X \otimes 1) = X \otimes 1$ for any $X \in \Uqg$.
For the second property, let $X_\mu$ be a lowest weight vector of $V$ of weight $\mu$.
Then $\tmapS(1 \otimes X_\mu) = K_\mu \otimes X_\mu$ by the previous result.
Given any $Y \in \UqlS$ we compute
\[
\tmapS(Y \triangleright (1 \otimes X_\mu))
= Y \blacktriangleright \tmapS(1 \otimes X_\mu)
= Y_{(1)} K_\mu S(Y_{(3)}) \otimes Y_{(2)} \triangleright X_\mu.
\]
Here we have used the equivariance of $\tmapS$ with respect to $\UqlS$ and the definition of the left action $\blacktriangleright$.
Since $X_\mu$ generates $V$ as a $\UqlS^{\geq 0}$-module, we obtain
\[
\tmapS(1 \otimes Y \triangleright X_\mu)
= \tmapS(Y \triangleright (1 \otimes X_\mu))
\in \UqlS^{\geq 0} \otimes V. \qedhere
\]
\end{proof}

We are going to use these results to analyze the coproducts of elements in $\qnil$.

\subsection{Degree one}

For the rest of this section we consider a subset of the form $S = \simpleroots \backslash \{\alpha_s\}$.
Then, as in \cref{prop:equivariant-quantizations}, we have a quantization $\qnil_S$ of the positive nilradical $\lien_S$, which is invariant under the left adjoint action of the quantized Levi factor $\UqlS$ and $\bbN$-graded by $\alpha_s$.
Our goal is to show that we can choose $\qnil_S$ in such a way that $\bbC \oplus \qnil_S$ is a left $\Uqb$-coideal.
Recall that the components $\qnil_{S, 1}$ and $\qnil_{S, 2}$ of degrees one and two are uniquely determined by \cref{cor:uniqueness-quantization}, so the only possible choices appear from degree $\geq 3$.

We begin by showing that $\bbC \oplus \qnil_{S, 1}$ is a left coideal.
We note that this result could be proven by more elementary methods, but we state it in such a form that is well-suited to performing a similar analysis for the components of higher degree.

\begin{proposition}
\label{prop:coideal-degree-one}
For any $X \in \lien^q_{S, 1}$ we have
\[
\Delta(X) = \tmapS(X \otimes 1 + 1 \otimes X).
\]
This implies the property $\Delta(\qnil_{S, 1}) \subseteq \Uqb \otimes (\bbC \oplus \qnil_{S, 1})$.
\end{proposition}

\begin{proof}
Since $S = \simpleroots \backslash \{\alpha_s\}$, we have that $E_s$ is a lowest weight vector for the left adjoint action of $\UqlS$, as in \cref{lem:action-degree-one}.
Then using \cref{prop:properties-T} we get
\[
\Delta(E_s) = E_s \otimes 1 + K_s \otimes E_s = \tmapS (E_s \otimes 1 + 1 \otimes E_s).
\]
This proves the first claim for the lowest weight vector $E_s$.

Next, recall that the coproduct satisfies the property $\Delta(x \triangleright y) = x \blacktriangleright \Delta(y)$ for any $x, y \in \Uqg$.
Then for any $Y \in \UqlS$ we obtain the identity
\[
\Delta(Y \triangleright E_s) = Y \blacktriangleright \Delta(E_s) = Y \blacktriangleright \tmapS (E_s \otimes 1 + 1 \otimes E_s).
\]
Similarly, the map $\tmapS$ satisfies the equivariance property $x \blacktriangleright \tmapS(y \otimes z) = \tmapS(x \triangleright (y \otimes z))$ for any $x \in \UqlS$ and $y, z \in \Uqg$.
Using this in the identity above, we get
\[
\Delta(Y \triangleright E_s) = \tmapS (Y \triangleright (E_s \otimes 1 + 1 \otimes E_s)) = \tmapS (Y \triangleright E_s \otimes 1 + 1 \otimes Y \triangleright E_s).
\]
Since $E_s$ generates $\qnil_{S, 1}$ as an $\UqlS$-module, this proves the first claim.
The second claim immediately follows from \cref{cor:action-T-simple}.
\end{proof}

This result shows that $\bbC \oplus \qnil_{S, 1}$ is a left $\Uqb$-coideal.
We note that here we could replace $\Uqb$ with $\UqlS^{\geq 0}$, but this is not going to be the case for the higher degrees.

Using the previous result we can also prove the following identity, which is going to be useful for studying the coideal property in general.

\begin{lemma}
\label{lem:delta-shuffle}
Let $X_1, \cdots, X_n \in \qnil_{S, 1}$. Then we have
\[
\Delta(X_1 \cdots X_n) = \tmapS \circ \pmap^{(n)} (X_1, \cdots, X_n),
\]
where the map $\pmap^{(n)}$ is defined in \cref{prop:shuffle-product}.
\end{lemma}

\begin{proof}
Clearly $\Delta(X_1 \cdots X_n) = \Delta(X_1) \cdots \Delta(X_n)$.
We have $\Delta(X) = \tmapS(X \otimes 1 + 1 \otimes X)$ for any $X \in \qnil_{S, 1}$ by \cref{prop:coideal-degree-one}.
From this we get
\[
\Delta(X_1 \cdots X_n) = \tmapS(X_1 \otimes 1 + 1 \otimes X_1) \cdots \tmapS(X_n \otimes 1 + 1 \otimes X_n).
\]
Now recall that $\tmapS$ is an algebra isomorphism from $\Uqg \otimes_\sigma \Uqg$ to $\Uqg \otimes \Uqg$, where for the first algebra we have the product $\cdot_\sigma$ which involves the braiding $\sigma = \sigma_S$ corresponding to $S$.
This means that we can rewrite
\[
\Delta(X_1 \cdots X_n) = \tmapS \left( (X_1 \otimes 1 + 1 \otimes X_1) \cdot_\sigma \cdots \cdot_\sigma (X_n \otimes 1 + 1 \otimes X_n) \right).
\]
Then the result follows from the definition of the map $\pmap^{(n)}$.
\end{proof}

\subsection{Degree two}

Next we consider the degree-two component $\lien^q_{S, 2} \subset \lien^q_{S, 1} \cdot \lien^q_{S, 1}$, which is uniquely determined as observed in \cref{cor:uniqueness-quantization}.
In the following result we make crucial use of the previously introduced map $\tmapS$, making the proof quite simple.

\begin{proposition}
\label{prop:coideal-degree-two}
We have
\[
\Delta(\lien^q_{S, 2}) \subseteq \Uqb \otimes (\bbC \oplus \qnil_{S, 1} \oplus \qnil_{S, 2}).
\]
\end{proposition}

\begin{proof}
Let $\{X_i\}_i$ be a basis of $\qnil_{S, 1}$.
Then we have $\Delta(X_i X_j) = \tmapS \circ \pmap^{(2)} (X_i, X_j)$ by \cref{lem:delta-shuffle}.
It easily follows from \cref{prop:shuffle-product} that
\[
\pmap^{(2)} (X_i, X_j) = X_i X_j \otimes 1 + (\id + \sigma) (X_i \otimes X_j) + 1 \otimes X_i X_j.
\]
Here we note that $(\id + \sigma) (X_i \otimes X_j) \in \qnil_{S, 1} \otimes \qnil_{S, 1} \subset \Uqb \otimes \qnil_{S, 1}$. By \cref{prop:properties-T} we know that $\tmapS$ maps $\Uqb \otimes \qnil_{S, 1}$ into itself, which gives
\[
\tmapS \circ (\id + \sigma) (X_i \otimes X_j) \in \Uqb \otimes \qnil_{S, 1}.
\]
Now consider any $X \in \qnil_{S, 2}$, which we can write in the form $X = \sum_{i, j} c_{i j} X_i X_j$. Then
\[
\tmapS(X \otimes 1) \in \qnil_{S, 2} \otimes \bbC, \quad
\tmapS(1 \otimes X) \in \Uqb \otimes \qnil_{S, 2},
\]
by \cref{cor:action-T-simple}. This gives the result.
\end{proof}

In other words, the subspace $\bbC \oplus \qnil_{S, 1} \oplus \qnil_{S, 2}$ is a left $\Uqb$-coideal.

\subsection{Degree three}

The case of degree three is significantly more complicated.
One issue is that $\qnil_{S, 3}$ is not uniquely determined in this case.
Indeed, we get at least two simple components of this type in $(\qnil_{S, 1})^{\otimes 3}$, arising from the decompositions of
\[
\qnil_{S, 2} \otimes \qnil_{S, 1}, \quad
\qnil_{S, 1} \otimes \qnil_{S, 2}.
\]
We begin by spelling out the product formula from \cref{prop:shuffle-product} in this case.

\begin{lemma}
\label{lem:shuffle-three}
For any $x_1, x_2, x_3 \in \Uqg$ we have
\[
\begin{split}
\pmap^{(3)} (x_1, x_2, x_3)
& = x_1 x_2 x_3 \otimes 1 + 1 \otimes x_1 x_2 x_3 \\
& + \mu_1 (1 + \sigma_2 + \sigma_2 \sigma_1) (x_1 \otimes x_2 \otimes x_3) \\
& + \mu_2 (1 + \sigma_1 + \sigma_1 \sigma_2) (x_1 \otimes x_2 \otimes x_3).
\end{split}
\]
\end{lemma}

\begin{proof}
The relevant shuffles $S_{i, j}$ with $i + j = 3$ are
\[
S_{0, 3} = \{ 1 \}, \quad
S_{1, 2} = \{ 1, s_1, s_2 s_1 \}, \quad
S_{2, 1} = \{ 1, s_2, s_1 s_2 \}, \quad
S_{3, 0} = \{ 1 \}.
\]
Taking their inverses and inserting them into $\pmap^{(3)}$ gives the result.
\end{proof}

We begin by giving an equivalent condition for $\qnil_{S, 3}$ to give rise to a left coideal.

\begin{lemma}
\label{lem:degree-three-property}
Let $\tilde{\lien}^q_{S, 3} \subset (\qnil_{S, 1})^{\otimes 3}$ be an $\UqlS$-module corresponding to the $\liel_S$-module $\lien_{S, 3}$.
Denote by $\qnil_{S, 3}$ its image in $\Uqn$ by the multiplication map.
Then we have
\[
\Delta(\lien^q_{S, 3}) \subseteq \Uqb \otimes (\bbC \oplus \qnil_{S, 1} \oplus \qnil_{S, 2} \oplus \qnil_{S, 3}).
\]
if and only if we have the property
\[
\mu_2 (1 + \sigma_1 + \sigma_1 \sigma_2) (\tilde{\lien}^q_{S, 3}) \subseteq \qnil_{S, 1} \otimes \qnil_{S, 2}.
\]
\end{lemma}

\begin{proof}
We denote by $\{X_i\}_i$ a basis of $\qnil_{S, 1}$, as before.
Then by \cref{lem:delta-shuffle} we have
\[
\Delta(X_i X_j X_k) = \tmapS \circ \pmap^{(3)} (X_i, X_j, X_k),
\]
where $\pmap^{(3)}$ is given explicitly in \cref{lem:shuffle-three}.
We note that
\[
\mu_1 (1 + \sigma_2 + \sigma_2 \sigma_1) (X_i \otimes X_j \otimes X_k) \in (\qnil_{S, 1})^2 \otimes \qnil_{S, 1} \subset \Uqb \otimes \qnil_{S, 1}.
\]
The same is true after applying $\tmapS$, by \cref{prop:properties-T}.
Now consider any element $X = \sum_{i, j, k} c_{i j k} X_i X_j X_k \in \qnil_{S, 3}$.
Then by \cref{cor:action-T-simple} we get
\[
\tmapS(X \otimes 1) \in \Uqb \otimes \bbC, \quad
\tmapS(1 \otimes X) \in \Uqb \otimes \qnil_{S, 3}.
\]
This accounts for three out of the four terms appearing in \cref{lem:shuffle-three}.

The only term left to consider is the one with the map $\mu_2 (1 + \sigma_1 + \sigma_1 \sigma_2)$. We have
\[
\mu_2 (1 + \sigma_1 + \sigma_1 \sigma_2) (\tilde{\lien}^q_{S, 3}) \subseteq \qnil_{S, 1} \otimes (\qnil_{S, 1})^2 \subset \Uqb \otimes (\qnil_{S, 1})^2.
\]
Decompose $(\qnil_{S, 1})^2$ into simple components, among which we have $\qnil_{S, 2}$ with multiplicity one.
Let $V$ be any such component and note that it must have degree two.
By \cref{prop:properties-T} we get $\tmapS(\Uqb \otimes V) \subseteq \Uqb \otimes V$.
Since $\tmapS$ is invertible, we have
\[
\tmapS(\Uqb \otimes V) \subset \Uqb \otimes (\bbC \oplus \qnil_{S, 1} \oplus \qnil_{S, 2} \oplus \qnil_{S, 3})
\]
if and only if $V = \qnil_{S, 2}$.
This gives the condition $\mu_2 (1 + \sigma_1 + \sigma_1 \sigma_2) (\tilde{\lien}^q_{S, 3}) \subseteq \qnil_{S, 1} \otimes \qnil_{S, 2}.$
\end{proof}

Hence the problem boils down to finding a simple component $\tilde{\lien}^q_{S, 3} \subset (\qnil_{S, 1})^{\otimes 3}$, corresponding to the $\liel_S$-module $\lien_{S, 3}$, such that the property from \cref{lem:degree-three-property} is satisfied.

As we discuss below, this is easy to achieve if $(\qnil_{S, 1})^{\otimes 3}$ has exactly two components of this type. However, we do not know of any a priori reason why this should be the case.
Hence we have to resort to some tedious case-by-case checking in \cref{lem:degree-three-multiplicity} (which is perhaps to be expected, since this situation only arises for exceptional Lie algebras).

\begin{proposition}
\label{prop:coideal-degree-three}
There exists $\lien^q_{S, 3} \subset \Uqn$ such that we have
\[
\Delta(\lien^q_{S, 3}) \subseteq \Uqb \otimes (\bbC \oplus \qnil_{S, 1} \oplus \qnil_{S, 2} \oplus \qnil_{S, 3}).
\]
Moreover this choice is uniquely determined.
\end{proposition}

\begin{proof}
To simplify the notation we write $\lieu_k = \lien_{S, k}$ for the classical $\liel_S$-modules and $\lieu^q_k = \qnil_{S, k}$ for the $\UqlS$-modules in $\Uqn$ (where $k = 1, 2, 3$).
Let us also write $\liec^q_3$ for the (abstract) $\UqlS$-module quantizing the classical $\liel_S$-module $\lieu_3$ for the decomposition under consideration.
Then we are looking for a (concrete) $\UqlS$-module $\lieu^q_3 \subset \Uqn$ satisfying the condition in the claim.
We have already determined the (unique) components $\lieu^q_1$ and $\lieu^q_2$ of degrees one and two.
We also write $\tilde{\lieu}^q_2 \subset \lieu^q_1 \otimes \lieu^q_1$ for the unique simple component such that $\mu(\tilde{\lieu}^q_2) = \lieu^q_2$.

Both tensor products $\tilde{\lieu}^q_2 \otimes \lieu^q_1$ and $\lieu^q_1 \otimes \tilde{\lieu}^q_2$ contain a unique simple component of type $\liec^q_3$, by the quantum analogue of \cref{prop:classical-module} (since both modules are weight multiplicity-free).
Hence we have at least two components of type $\liec^q_3$ inside $(\lieu^q_1)^{\otimes 3}$, which we write as
\[
C^q_1 \subset \tilde{\lieu}^q_2 \otimes \lieu^q_1 \subset (\lieu^q_1)^{\otimes 3}, \quad
C^q_2 \subset \lieu^q_1 \otimes \tilde{\lieu}^q_2 \subset (\lieu^q_1)^{\otimes 3}.
\]
Denote the corresponding $\UqlS$-module maps by $\iota^q_1$ and $\iota^q_2$, so that $C^q_n = \iota^q_n(\liec^q_3)$ for $n = 1, 2$.

By the quantum analogue of \cref{lem:degree-three-multiplicity}, there are exactly two components of type $\liec^q_3$ inside $(\lieu^q_1)^{\otimes 3}$.
Hence any $\UqlS$-module map $\liec^q_3 \to (\lieu^q_1)^{\otimes 3}$ can be written as a linear combination of $\iota^q_1$ and $\iota^q_2$. Equivalently, any $\UqlS$-module isomorphic to $\liec^q_3$ in $(\lieu^q_1)^{\otimes 3}$ is of the form
\[
C^q_{c_1, c_2} = \iota^q_{c_1, c_2} (\liec^q_3), \quad
\iota^q_{c_1, c_2} = c_1 \iota^q_1 + c_2 \iota^q_2,
\]
provided the map $\iota^q_{c_1, c_2}$ is non-zero (which is the case if either $c_1$ or $c_2$ is non-zero).

We claim that the image of $C^q_{c_1, c_2}$ in $\Uqn$ is non-zero when $\iota^q_{c_1, c_2}$ is non-zero (for $q$ an indeterminate or transcendental).
To see this we consider the specialization at $q = 1$ (for which we drop the decoration $q$ everywhere). We can choose $\iota_2$ and $\iota_1$ such that $\iota_2 = \tau_1 \tau_2 \iota_1$, where $\tau$ denotes the flip map.
Then $C_{c_1, c_2} \subset \lieu_1^{\otimes 3}$ is of the form
\[
c_1 \sum_i \tilde{y}_i \otimes x_i + c_2 \sum_i x_i \otimes \tilde{y}_i,
\]
where $\sum_i \tilde{y}_i \otimes x_i \in C_1 \subset \tilde{\lieu}_2 \otimes \lieu_1$.
Writing $y_i = \mu(\tilde{y}_i)$ for the image of $\tilde{y}_i$ in $U(\lien)$, we see that the image $\mu_1^2(C_{c_1, c_2})$ of $C_{c_1, c_2}$ in $U(\lien)$ consists of elements of the form
\[
c_1 \sum_i y_i x_i + c_2 \sum_i x_i y_i = (c_1 + c_2) \sum_i y_i x_i + c_2 \sum_i [x_i, y_i].
\]
Recall that $[\lieu_1, \lieu_2] = \lieu_3$ by \cref{prop:classical-module}.
Using this and the PBW theorem for $U(\lien)$, it follows that the image of $C_{c_1, c_2}$ is non-zero, provided either $c_1$ or $c_2$ is non-zero.
Hence any $\UqlS$-module $\lieu^q_3 \subset (\lieu^q_1)^3$ is of the form $\lieu^q_3 = \mu_1^2 (C^q_{c_1, c_2})$ for some $c_1$ and $c_2$.

Next, we investigate for which choice of coefficients $c_1$ and $c_2$ this gives rise to a left coideal.
By \cref{lem:degree-three-property} this is the case if and only if
\[
\mu_2 (1 + \sigma_1 + \sigma_1 \sigma_2) (C^q_{c_1, c_2}) \subset \lieu^q_1 \otimes \lieu^q_2.
\]
Note that $\mu_2(C^q_2) \subset \lieu^q_1 \otimes \lieu^q_2$ by definition of $C^q_2$, and this is the only component of type $\liec^q_3$ in $\lieu^q_1 \otimes \lieu^q_2$.
Since $\mu_2 (1 + \sigma_1 + \sigma_1 \sigma_2)$ is a $\UqlS$-module map, the condition above becomes
\begin{equation}
\label{eq:coideal-condition}
\mu_2 (1 + \sigma_1 + \sigma_1 \sigma_2) (C^q_{c_1, c_2}) \subseteq \mu_2(C^q_2).
\end{equation}

We can study this condition in terms the maps $\iota^q_1$ and $\iota^q_2$.
Let us write
\[
(1 + \sigma_1 + \sigma_1 \sigma_2) \iota^q_n = t_{n, 1} \iota^q_1 + t_{n, 2} \iota^q_2 = \iota^q_{t_{n, 1}, t_{n, 2}}, \quad n = 1, 2.
\]
Then $(1 + \sigma_1 + \sigma_1 \sigma_2) \iota^q_{c_1, c_2} = \iota^q_{d_1, d_2}$, where $d_n = c_1 t_{1, n} + c_2 t_{2, n}$ for $n = 1, 2$.
In particular we have $(1 + \sigma_1 + \sigma_1 \sigma_2) \iota^q_{c_1, c_2} (\liec^q_3) = C^q_{d_1, d_2}$.
As discussed above, the image of $C^q_{d_1, d_2}$ in $\Uqn$ is non-zero if either $d_1$ or $d_2$ is non-zero.
In particular, it follows that the $\UqlS$-module maps $\mu_2 \iota_1$ and $\mu_2 \iota_2$ are linearly independent.
Then let us consider
\[
\mu_2 (1 + \sigma_1 + \sigma_1 \sigma_2) \iota_{c_1, c_2} = d_1 \mu_2 \iota^q_1 + d_2 \mu_2 \iota^q_2.
\]
Since $\mu_2(C^q_2) = \mu_2 \iota_2 (\liec^q_3)$, we see that \eqref{eq:coideal-condition} is satisfied if and only if $d_1 = c_1 t_{1, 1} + c_2 t_{2, 1} = 0$.
We can find a solution with either $c_1$ or $c_2$ non-zero, which takes care of the existence of a $\UqlS$-module $\lieu^q_3 \subset \Uqn$ satisfying the condition on the coproduct.
Hence the only thing left to show is uniqueness of $\lieu^q_3$, for which we need some additional preparation.

The braiding $\sigma$ acts on the simple component $\tilde{\lieu}^q_2 \subset \lieu^q_1 \otimes \lieu^q_1$ by a scalar, by Schur's lemma.
Its value can be computed along the lines of \cite[Section 8.4, Corollary 23]{klsc} and takes the form $- q^s$ for some rational number $s$.
Here the minus sign corresponding to the classical antisymmetry of the component (as in \cref{prop:classical-module}), while the exponent $s$ can be computed explicitly in terms of Lie algebraic data (in our conventions we have $s = c(\lambda_2) / 2 - c(\lambda_1)$, where $\lambda_1$ and $\lambda_2$ are the highest weights of $\lieu_1$ and $\lieu_2$, and $c(\lambda) = (\lambda, \lambda + 2 \rho)$ is the classical value of the Casimir).
All we need to know here is that $s \neq 0$ for the cases under consideration, as in \cref{lem:degree-three-multiplicity} (we omit the details of this verification).

Given this fact, it follows from the definition of $C^q_1$ and $C^q_2$ that we have
\[
\sigma_1 \iota^q_1 = - q^s \iota^q_1, \quad
\sigma_2 \iota^q_2 = - q^s \iota^q_2.
\]
Using the braid equation, it is not difficult to show that $\sigma_1 \sigma_2$ maps $C^q_1$ to $C^q_2$ and $\sigma_2 \sigma_1$ maps $C^q_2$ to $C^q_1$.
Then there are non-zero coefficients $a_1$ and $a_2$ such that
\[
\sigma_1 \sigma_2 \iota^q_1 = a_1 \iota^q_2, \quad
\sigma_2 \sigma_1 \iota^q_2 = a_2 \iota^q_1.
\]
Finally let us write $\sigma_1 \iota^q_2 = b_1 \iota^q_1 + b_2 \iota^q_2$ for the decomposition of this $\UqlS$-module map.

Then, taking into account these relations, we find the identities
\[
\begin{split}
(1 + \sigma_1 + \sigma_1 \sigma_2) \iota^q_1
& = (1 - q^s) \iota^q_1 + a_1 \iota^q_2, \\
(1 + \sigma_1 + \sigma_1 \sigma_2) \iota^q_2
& = (1 - q^s) b_1 \iota^q_1 + (1 + (1 - q^s) b_2) \iota^q_2.
\end{split}
\]
Using the previously introduced notation, this means that $t_{1, 1} = 1 - q^s$ and $t_{2, 1} = (1 - q^s) b_1$.
Recall that we need to satisfy the condition $d_1 = c_1 t_{1, 1} + c_2 t_{2, 1} = 0$.
We see that this can be written as $(1 - q^s) (c_1 + c_2 b_1) = 0$, which is equivalent to $c_1 + c_2 b_1 = 0$.

If $b_1 \neq 0$ then both $c_1$ and $c_2$ must be non-zero (since we are excluding the case $c_1 = c_2 = 0$). In this case $C^q_{c_1, c_2} = C^q_{- b_1, 1}$ is uniquely determined.
Suppose on the contrary that $b_1 = 0$, so that $\sigma_1 \iota^q_2 = b_2 \iota^q_2$. Then we have $\sigma_2 \sigma_1 \iota^q_2 = b_2 \sigma_2 \iota^q_2 = - q^s b_2 \iota^q_2$.
On the other hand, we have $\sigma_2 \sigma_1 \iota^q_2 = a_2 \iota_1$ with $a_2 \neq 0$. This would imply that $\iota^q_1$ and $\iota^q_2$ are linearly dependent, which is a contradiction (another possibility it to use specialization at $q = 1$).
\end{proof}

\subsection{Main result}

We now come back to the equivariant quantizations $\qnil_1 \oplus \cdots \oplus \qnil_r$ from \cref{prop:equivariant-quantizations} to investigate the corresponding coideal property.
We recall that each summand $\qnil_k$ is $\bbN$-graded by some simple root $\alpha_{i_k}$.
In the next result we show that, if each $\qnil_k$ has at most three graded components, then it can be uniquely chosen in such a way that $\bbC \oplus \qnil_k$ is a left $\Uqb$-coideal.
This restriction can be reformulated in terms of the simple components appearing in the semisimple Lie algebra $\lieg$, as we do below.

\begin{theorem}
\label{thm:coideal-quantizations}
Let $\lieg$ be a complex semisimple Lie algebra of rank $r$, which does not contain components of type $F_4$, $E_7$, $E_8$.
Let $S_1 \subset \cdots \subset S_r$ be an increasing sequence of proper subsets of $\simpleroots$, with $\lien = \lien_1 \oplus \cdots \oplus \lien_r$ the corresponding decomposition of the positive nilradical.

Then there exists a unique finite-dimensional subspace
\[
\qnil = \qnil_1 \oplus \cdots \oplus \qnil_r \subset \Uqn
\]
such that each summand $\qnil_k$ satisfies the following properties:
\begin{itemize}
\item $\qnil_k$ is an $U_q(\liel_k)$-module under $\triangleright$, corresponding to the $\liel_k$-module $\lien_k$,
\item each graded component $\qnil_{k, n}$ is contained in the $n$-fold product of $\qnil_{k, 1}$,
\item $\bbC \oplus \qnil_k$ is a left $\Uqb$-coideal, that is $\Delta(\bbC \oplus \qnil_k) \subset \Uqb \otimes (\bbC \oplus \qnil_k)$.
\end{itemize}
\end{theorem}

\begin{proof}
We have already shown in \cref{prop:equivariant-quantizations} that an equivariant quantization of $\lien = \lien_1 \oplus \cdots \oplus \lien_r$ exists (but is not unique, in general), that is the first two properties are satisfied.
For the coideal property we first recall that, as given in \eqref{eq:lie-algebra-inclusion}, the sequence $S_1 \subset \cdots \subset S_r$ determines a sequence of complex semisimple Lie algebras
\[
\{0\} = \lieg_1 \subset \lieg_2 \subset \cdots \subset \lieg_r \subset \lieg_{r + 1} = \lieg.
\]
We have that $\lien_k \subset \lieg_{k + 1}$ is the positive nilradical corresponding to the inclusion $S_k \subset S_{k + 1}$, more precisely $S_k = S_{k + 1} \backslash \{ \alpha_{i_k} \}$.
In other words, each summand $\lien_k$ appears as the positive nilradical $\lien_S$ of a semisimple Lie algebra corresponding to a proper subset $S$ obtained by removing one simple root, which is the setting we considered above.

If $\lieg$ does not contain simple components of type $F_4$, $E_7$, $E_8$, then each summand $\lien_k$ has graded components of degree at most three, as can be easily deduced by inspecting the highest roots of the various complex simple Lie algebras.
Then by combining \cref{prop:coideal-degree-one}, \cref{prop:coideal-degree-two} and \cref{prop:coideal-degree-three}, which correspond respectively to the cases of degree one, two and three, we deduce that $\qnil_k$ can be chosen in such a way that $\bbC \oplus \qnil_k$ is a left $\Uqb$-coideal.
This choice is necessarily unique in degrees one and two, while \cref{prop:coideal-degree-three} guarantees that this is the case in degree three as well.
\end{proof}

This result refines \cref{prop:equivariant-quantizations} and gives a satisfactory quantization of the decomposition $\lien = \lien_1 \oplus \cdots \oplus \lien_r$ together with the action of the various Levi factors.
Indeed, under the conditions stated in the result, we have a unique quantization within $\Uqn$ satisfying an analogue of the existence of a Lie algebra structure on each summand $\lien_k$.

We strongly suspect that it should be possible to remove the restriction that $\lieg$ should not contain simple components of type $F_4$, $E_7$ and $E_8$.
Doing so would require the analysis of the graded components $\qnil_{k, n}$ up to degree $n = 6$ (but we note that the case of degree four would suffice for $F_4$ and $E_7$).
In principle this analysis can performed following the strategy employed in the degree-three case, but the details get significantly harder.
In particular, it would be desirable to have a more conceptual understanding for the multiplicity of $\lien_{k, n}$ inside the tensor product $\lien_{k, 1}^{\otimes n}$, rather than resorting to explicit case-by-case checks as in \cref{lem:degree-three-multiplicity}.
In any case, this leads us to formulate the following conjecture.

\begin{conjecture}
\label{conj:coideal}
The results of \cref{thm:coideal-quantizations} hold for any finite-dimensional complex semisimple Lie algebra $\lieg$, without any restriction on the simple components.
\end{conjecture}

As partial evidence beyond \cref{thm:coideal-quantizations}, we point out that for any semisimple Lie algebra $\lieg$ of rank $r$ we have the following: for any $k \in \{ 1, \cdots, r \}$ and $n \in \{ 1, 2, 3 \}$, we have that $\bbC \oplus \qnil_{k, 1} \oplus \cdots \oplus \qnil_{k, n}$ is a left $\Uqb$-coideal.
Indeed, in the proof of \cref{prop:coideal-degree-three} there was no restriction on the semisimple Lie algebra being considered.

\section{The example of $G_2$}
\label{sec:example-G2}

In this section we illustrate the results obtained above in the case of the exceptional Lie algebra $G_2$.
In particular, we consider the decomposition $\lien = \lien_1 \oplus \lien_2$ such that $\lien_2$ has graded components up to degree three, which is the most complicated case for the purpose of checking the coideal property.
This is well-suited to illustrating the complexity of the situation, while keeping computations manageable.
In fact, we are going to skip most of the relevant details, only giving indications for how they can be performed.

\subsection{Decomposition}

We consider the simple Lie algebra $\lieg = G_2$.
We denote by $\alpha_1$ the short root and by $\alpha_2$ the long root, which means that
\[
(\alpha_1, \alpha_2^\vee) = -1, \quad
(\alpha_2, \alpha_1^\vee) = -3.
\]
The corresponding positive roots are given by
\[
\roots^+ = \{ \alpha_1, \ \alpha_2, \ \alpha_1 + \alpha_2, \ 2 \alpha_1 + \alpha_2, \ 3 \alpha_1 + \alpha_2, \ 3 \alpha_1 + 2 \alpha_2 \}.
\]

We consider the decomposition $\lien = \lien_1 \oplus \lien_2$ with respect to the sequence $S_1 \subset S_2$ with $S_2 = \{ \alpha_2 \}$.
The corresponding Levi factors $\liel_1$ and $\liel_2$ are given by
\[
\liel_1 = \lieh, \quad
\liel_2 = \lieh \oplus \lspan \{ e_{\alpha_2}, f_{\alpha_2} \}.
\]
We note that the semisimple part of $\liel_2$ can be identified with $\mathfrak{sl}_2$.

For the positive nilradical we have $\lien_1 = \bbC e_{\alpha_2}$ and $\lien_2 = \lien_{2, 1} \oplus \lien_{2, 2} \oplus \lien_{2, 3}$, where the grading in this case is with respect to $\alpha_1$. These three components are given by
\[
\lien_{2, 1} = \lspan \{ e_{\alpha_1 + \alpha_2}, e_{\alpha_1} \}, \quad
\lien_{2, 2} = \lspan \{ e_{2 \alpha_1 + \alpha_2} \}, \quad
\lien_{2, 3} = \lspan \{ e_{3 \alpha_1 + 2 \alpha_2}, e_{3 \alpha_1 + \alpha_2} \}.
\]
These are all $\mathfrak{sl}_2$-modules by the adjoint action of the semisimple part of $\liel_2$. We note that the components of degrees one and three correspond to the fundamental module, while the component of degree two corresponds to the trivial module.

\subsection{Quantum root vectors}

Our goal is to exhibit explicitly the equivariant quantization $\qnil = \qnil_1 \oplus \qnil_2$ corresponding to the decomposition above.
We begin by introducing quantum root vectors that are adapted to this situation, as in \cref{prop:schubert-component}.
This means that the longest word $w_0$ should be factorized as $w_0 = w_{0, S_2} w_{S_2}$ where $S_2 = \{ \alpha_2 \}$.
In other words, $w_0$ should begin with the simple reflection $s_2$, giving  the following enumeration.

\begin{lemma}
Consider $w_0 = s_2 s_1 s_2 s_1 s_2 s_1$. The corresponding positive roots are
\[
\begin{gathered}
\beta_1 = \alpha_2, \
\beta_2 = \alpha_1 + \alpha_2, \
\beta_3 = 3 \alpha_1 + 2 \alpha_2, \\
\beta_4 = 2 \alpha_1 + \alpha_2, \ 
\beta_5 = 3 \alpha_1 + \alpha_2, \
\beta_6 = \alpha_1.
\end{gathered}
\]
\end{lemma}

We define the quantum root vectors $E_{\beta_1}, \cdots, E_{\beta_6}$ with respect to $w_0$ as above.
We also note that, with respect to this enumeration, the decomposition of $\lien$ can be written as
\[
\lien_{1, 1} = \lien_{\beta_1}, \quad
\lien_{2, 1} = \lien_{\beta_2} \oplus \lien_{\beta_6}, \quad
\lien_{2, 2} = \lien_{\beta_4}, \quad
\lien_{2, 3} = \lien_{\beta_3} \oplus \lien_{\beta_5}.
\]
It follows from \cref{prop:schubert-component} that the degree-one components are given by
\[
\qnil_{1, 1} = \lspan \{ E_{\beta_1} \}, \quad
\qnil_{2, 1} = \lspan \{ E_{\beta_2}, E_{\beta_6} \}.
\]
So it remains to find the higher-degree components $\qnil_{2, 2}$ and $\qnil_{2, 3}$.

Next, we are going to determine the $q$-commutators $[E_{\beta_i}, E_{\beta_j}]_q = E_{\beta_i} E_{\beta_j} - q^{(\beta_i, \beta_j)} E_{\beta_j} E_{\beta_i}$ in the case $i < j$.
We recall that, in our conventions for the coproduct, we have the identity $E_i \triangleright X = [E_i, X]_q$ for the generators of the quantized Levi factor, as in \cref{prop:adjoint-action-schubert}.

In the next result we use some standard notations related to quantized enveloping algebras, such as the $q$-numbers $[n]_q := \frac{q^n - q^{-n}}{q - q^{-1}}$ and the divided powers $E_\beta^{(n)} := E_\beta^n / [n]_{q_\beta}!$.

\begin{lemma}
\label{lem:G2-qcommutators}
The non-zero $q$-commutators $[E_{\beta_i}, E_{\beta_j}]_q$ with $i < j$ are given by
\[
\begin{gathered}
[E_{\beta_1}, E_{\beta_3}]_q = (q^6 - q^4 - q^2 + 1) E_{\beta_2}^{(3)}, \quad
[E_{\beta_1}, E_{\beta_4}]_q = (q^3 - q^{-1}) E_{\beta_2}^{(2)}, \\
[E_{\beta_1}, E_{\beta_5}]_q = - (q + q^{-1}- q^{-3}) E_{\beta_3} + (q - q^{-1}) E_{\beta_2} E_{\beta_4}, \\
[E_{\beta_1}, E_{\beta_6}]_q = E_{\beta_2}, \quad
[E_{\beta_2}, E_{\beta_4}]_q = [3]_q E_{\beta_3}, \quad
[E_{\beta_2}, E_{\beta_5}]_q = (q^3 - q^{-1}) E_{\beta_4}^{(2)}, \\
[E_{\beta_2}, E_{\beta_6}]_q = [2]_q E_{\beta_4}, \quad
[E_{\beta_3}, E_{\beta_5}]_q = (q^6 - q^4 - q^2 + 1) E_{\beta_4}^{(3)}, \\
[E_{\beta_3}, E_{\beta_6}]_q = (q^3 - q^{-1}) E_{\beta_4}^{(2)}, \quad
[E_{\beta_4}, E_{\beta_6}]_q = [3]_q E_{\beta_5}.
\end{gathered}
\]
\end{lemma}

\begin{proof}
These relations can be checked using the package \emph{QuaGroup} available in GAP \cite{GAP4}.
We warn the reader that the reduced decomposition $w_0 = s_2 s_1 s_2 s_1 s_2 s_1$ is not the default one for this package, so that it should be specified to reproduce the results above.
\end{proof}

Using the $q$-commutators above, we see explicitly that $\qnil_{2, 1} = \lspan \{ E_{\beta_2}, E_{\beta_6} \}$ is invariant under the left adjoint action of (the positive part of) $U_q(\liel_2)$, in agreement with the general result from \cref{prop:adjoint-action-schubert}.
On the other hand, we also see that $[E_{\beta_1}, E_{\beta_4}]_q \neq 0$ in the quantum setting, which means that $E_{\beta_4}$ is not a highest weight vector for the left adjoint action of $U_q(\liel_2)$. It follows that $E_{\beta_4}$ is not going to be an element of $\qnil_{2, 2}$.
Similar remarks apply to $E_{\beta_3}$ and $E_{\beta_5}$ with respect to the degree-three component $\qnil_{2, 3}$.

\subsection{Braiding and coideal}

Before discussing the higher-degree components $\qnil_{2, 2}$ and $\qnil_{2, 3}$, we determine the braiding on the degree-one component $\qnil_{2, 1} = \lspan \{ E_{\beta_2}, E_{\beta_6} \}$.

\begin{lemma}
The braiding $\sigma: \qnil_{2, 1} \otimes \qnil_{2, 1} \to \qnil_{2, 1} \otimes \qnil_{2, 1}$ is given by
\[
\begin{split}
\sigma(E_{\beta_2} \otimes E_{\beta_2}) & = q^2 E_{\beta_2} \otimes E_{\beta_2}, \\
\sigma(E_{\beta_2} \otimes E_{\beta_6}) & = q^{-1} E_{\beta_6} \otimes E_{\beta_2} + (q^3 - q^{-3}) q^{-1} E_{\beta_2} \otimes E_{\beta_6}, \\
\sigma(E_{\beta_6} \otimes E_{\beta_2}) & = q^{-1} E_{\beta_2} \otimes E_{\beta_6}, \\
\sigma(E_{\beta_6} \otimes E_{\beta_6}) & = q^2 E_{\beta_6} \otimes E_{\beta_6}.
\end{split}
\]
\end{lemma}

\begin{proof}
Follows from routine computations.
Note that in our conventions for the braiding $\sigma$ we have $\sigma(E_{\beta_i} \otimes E_{\beta_2}) = q^{(\beta_i, \beta_2)} E_{\beta_2} \otimes E_{\beta_i}$ and $\sigma(E_{\beta_6} \otimes E_{\beta_i}) = q^{(\beta_6, \beta_i)} E_{\beta_i} \otimes E_{\beta_6}$, since $E_{\beta_2}$ and $E_{\beta_6}$ are respectively the highest and lowest weight vectors.
\end{proof}

To determine the component $\qnil_{2, 2}$ of degree two, we look for a simple component corresponding to the trivial module $\lien_{2, 2}$ inside the tensor product $\qnil_{2, 1} \otimes \qnil_{2, 1}$.
Also recall that this component is uniquely determined, as discussed in \cref{cor:uniqueness-quantization}.

\begin{lemma}
The element
\[
Y_{\beta_4} = E_{\beta_2} \otimes E_{\beta_6} - q^3 E_{\beta_6} \otimes E_{\beta_2}
\]
is invariant under the left adjoint action of $U_q(\liel_2)$ and satisfies $\sigma(Y_{\beta_4}) = -q^{-4} Y_{\beta_4}$.
\end{lemma}

Next, we denote by $C^q_1$ and $C^q_2$ the two simple components corresponding to $\lien_{2, 3}$ inside the tensor product $(\qnil_{2, 1})^{\otimes 3}$, using the same notation as in the proof of \cref{prop:coideal-degree-three}.
Then it is immediate to check that we have the lowest weight vectors
\[
L_1 = Y_{\beta_4} \otimes E_{\beta_6} \in C^q_1, \quad
L_2 = E_{\beta_6} \otimes Y_{\beta_4} \in C^q_2.
\]
As discussed in \cref{prop:coideal-degree-three}, we have that $1 + \sigma_1 + \sigma_1 \sigma_2$ maps $C^q_1 \oplus C^q_2$ into itself.
More precisely, we have the following identities.

\begin{lemma}
We have
\[
\begin{split}
(1 + \sigma_1 + \sigma_1 \sigma_2) (L_1) & = (1 - q^{-4}) L_1 + q L_2, \\
(1 + \sigma_1 + \sigma_1 \sigma_2) (L_2) & = (q^{-1} - q^{-5}) L_1 + (q^2 + 1 - q^{-2}) L_2.
\end{split}
\]
\end{lemma}

In particular, using these identities we obtain
\[
(1 + \sigma_1 + \sigma_1 \sigma_2) (L_1 - q L_2) = - (q^3 - q^{-1}) L_2 \in C^q_2.
\]
Denote by $\qnil_{2, 3}$ the $U_q(\liel_2)$-module generated by the lowest weight vector $L_1 - q L_2$.
Then \cref{prop:coideal-degree-three} guarantees that $\bbC \oplus \qnil_2$ is a left coideal.

\subsection{Results}

We summarize the computations above in the next proposition.

\begin{proposition}
Define the elements
\begin{align*}
X_{\beta_1} & = E_{\beta_1}, &
X_{\beta_2} & = E_{\beta_2}, &
X_{\beta_3} & = E_{\beta_1} X_{\beta_5} - q^{-3} X_{\beta_5} E_{\beta_1}, \\
X_{\beta_4} & = E_{\beta_2} E_{\beta_6} - q^3 E_{\beta_6} E_{\beta_2}, &
X_{\beta_5} & = X_{\beta_4} E_{\beta_6} - q E_{\beta_6} X_{\beta_4}, &
X_{\beta_6} & = E_{\beta_6}.
\end{align*}
Furthermore set $\qnil_1 = \qnil_{1, 1}$ and $\qnil_2 = \qnil_{2, 1} \oplus \qnil_{2, 2} \oplus \qnil_{2, 3}$, where
\begin{align*}
\qnil_{1, 1} & = \lspan \{ X_{\beta_1} \}, &
\qnil_{2, 1} & = \lspan \{ X_{\beta_2}, X_{\beta_6} \}, \\
\qnil_{2, 2} & = \lspan \{ X_{\beta_4} \}, &
\qnil_{2, 3} & = \lspan \{ X_{\beta_3}, X_{\beta_5} \}.
\end{align*}
Then $\qnil = \qnil_1 \oplus \qnil_2$ is the unique equivariant quantization from \cref{thm:coideal-quantizations} corresponding to the sequence $S_1 \subset S_2$ with $S_2 = \{ \alpha_2 \}$.
\end{proposition}

\begin{proof}
Follows from the computations illustrated above.
Here we note that $X_{\beta_5} = L_1 - q L_2$ is the lowest weight vector of $\qnil_{2, 3}$, while the highest weight vector $X_{\beta_3}$ is obtained using the left adjoint action as $X_{\beta_3} = E_2 \triangleright X_{\beta_5} = E_{\beta_1} X_{\beta_5} - q^{-3} X_{\beta_5} E_{\beta_1}$.
\end{proof}

We note that in the classical limit $X_{\beta_3}$, $X_{\beta_4}$, $X_{\beta_5}$ reduce to commutators of root vectors, as it should be for the positive nilradical $\lien$.
On the other hand, we note that in the quantum setting these are \emph{not} given by $q$-commutators, although they take a similar form.
For example, we see that $[E_{\beta_2}, E_{\beta_6}]_q = E_{\beta_2} E_{\beta_6} - q^{-1} E_{\beta_6} E_{\beta_2}$ does not coincide with $X_{\beta_4}$.

Furthermore, we see that $\qnil$ consists of three elements belonging to the chosen PBW-basis, namely $E_{\beta_1}$, $E_{\beta_2}$, $E_{\beta_6}$, together with three other elements that are not of this type.
Below we give their expressions in terms of the PBW-basis.

\begin{corollary}
With respect to the chosen basis of $\Uqn$, we have
\[
\begin{gathered}
X_{\beta_3} = - q^3 [3]_q! E_{\beta_3} + (q - q^{-1}) q^2 [2]_q E_{\beta_2} E_{\beta_4}, \quad
X_{\beta_4} = q^4 [2]_q E_{\beta_4} - (q^4 - 1) E_{\beta_2} E_{\beta_6}, \\
X_{\beta_5} = q^4 [3]_q! E_{\beta_5} - (q - q^{-1}) q^4 [2]_q^2 E_{\beta_4} E_{\beta_6} + (q - q^{-1})^2 q^3 [2]_q^2 E_{\beta_2} E_{\beta_6}^{(2)}.
\end{gathered}
\]
\end{corollary}

\begin{proof}
These can be obtained by using the relations from \cref{lem:G2-qcommutators}.
\end{proof}

We observe that they are all given as $\bbZ[q, q^{-1}]$-linear combinations of the fixed PBW-basis, which suggests some integrality property of this construction.

\section{Covariant differential calculi}
\label{sec:covariant-calculi}

Consider the unique equivariant quantization $\qnil = \qnil_1 \oplus \cdots \oplus \qnil_r$ corresponding to the sequence $S_1 \subset \cdots \subset S_r$ constructed in the previous sections.
We show how this construction can be used to produce covariant first-order differential calculi on the quantum flag manifolds $\Cq[G / P_{S_k}]$, in a manner compatible with the fixed sequence.
Moreover, we show that these calculi reduce to those introduced by Heckenberger-Kold in the irreducible case.

\subsection{Quantum homogeneous spaces}

We briefly recall the class of quantum homogeneous spaces considered in \cite{heko-homogeneous}, although we are only concerned with the special case of quantum flag manifolds, which are defined with respect to Hopf subalgebras.

Let $U$ be a Hopf algebra over $\bbC$ with bijective antipode and $K \subset U$ be a right coideal subalgebra.
Consider a tensor category $\calC$ of finite-dimensional left $U$-modules.
Let $\calA = U^\circ_\calC$ denote the dual Hopf algebra generated by the matrix coefficients of all left $U$-modules in $\calC$. Assume that $\calA$ separates the elements of $U$ and that its antipode is bijective.

In this setting, we define a left coideal subalgebra $\calB \subset \calA$ by
\[
\calB = \{ b \in \calA : b_{(1)} b_{(2)}(k) = \varepsilon(k) b, \ \forall k \in K \}.
\]
We assume furthermore that $K$ is $\calC$-semisimple, which implies that $\calA$ is a faithfully flat $\calB$-module. Then we call $\calB$ a \emph{quantum homogeneous space}.

\subsection{Differential calculi}

The notion of differential calculus on an algebra is fairly standard.
Here we are only concerned with the first-order part of such a calculus, as follows.

\begin{definition}
A \emph{first-order differential calculus} (FODC) over an algebra $\calB$ is a $\calB$-bimodule $\Gamma$ together with a linear map $\diff: \calB \to \Gamma$ such that $\Gamma = \lspan \{ a \diff b c : a, b, c \in \calB \}$ and $\diff$ satisfies the Leibnitz rule $\diff(a b) = \diff a b + a \diff b$.
\end{definition}

Suppose furthermore that $\calA$ is a Hopf algebra and that $\calB$ is a left $\calA$-comodule algebra, where we write $\Delta_{\calB}: \calB \to \calA \otimes \calB$ for the coaction.
In this setting we have the notion of covariant differential calculus, as defined by Woronowicz \cite{woronowicz}.

\begin{definition}
With notation as above, a FODC $\Gamma$ over $\calB$ is \emph{left-covariant} if $\Gamma$ possesses the structure of a left $\calA$-comodule $\Delta_\Gamma: \Gamma \to \calA \otimes \Gamma$ such that
\[
\Delta_\Gamma(a \diff b c) = \Delta_{\calB}(a) \cdot (\id \otimes \diff) \Delta_{\calB}(b) \cdot \Delta_{\calB}(c).
\]
\end{definition}

In the classical setting this corresponds to $\calA = \bbC[G]$ and $\calB = \bbC[G / K]$, with the coaction arising by dualizing the corresponding to the action of $G$ on forms of $G / K$.

\subsection{Tangent spaces}

By a result of Hermisson \cite{hermisson}, (finite-dimensional) left-covariant FODCs over $\calB$ are in one-to-one correspondence with certain right ideals of $\calB$.
For various reasons, it is more convenient to work with the corresponding \emph{quantum tangent spaces}, which are certain subspaces of the dual coalgebra $\calB^\circ$, as discussed in \cite[Section 4]{heko-homogeneous}.

Rather than spelling out their precise definition, we recall the following result.

\begin{proposition}[{\cite[Corollary 5]{heko-homogeneous}}]
\label{prop:calculi-correspondence}
Let $\calB \subset \calA$ be as described above. Then there is a canonical one-to-one correspondence between $n$-dimensional left-covariant FODCs over $\calB$ and $(n + 1)$-dimensional subspaces $T^\varepsilon \subset \calB^\circ$ such that
\[
\varepsilon \in T^\varepsilon, \quad
\Delta(T^\varepsilon) \subset T^\varepsilon \otimes \calB^\circ, \quad
K T^\varepsilon \subset T^\varepsilon.
\]
\end{proposition}

With respect to this correspondence, the quantum tangent space to a given FODC is the subspace of $T^\varepsilon$ consisting of elements vanishing at the unit of $\calB$.

Furthermore, this result shows that the classification of finite-dimensional left-covariant FODCs over $\calB$ requires the determination of the vector space $\{ f \in \calB^\circ : \dim K f < \infty \}$, which is called the \emph{locally-finite part} of the dual coalgebra $\calB^\circ$ in \cite{heko-homogeneous}.

In general it can be fairly complicated to work with the dual coalgebra $\calB^\circ$.
It would be more convenient to work with subspaces of $U$ satisfying analogous properties to \cref{prop:calculi-correspondence}, where we note that each $u \in U$ can be considered as a linear functional on $\calB$ by $b \mapsto b(u)$.
We can do so at the cost of losing the one-to-one correspondence, so that different subspaces $T$ and $T'$ of $U$ might give rise to the same left-covariant FODC over $\calB$.
Since this is not an issue for our purposes, we consider this setting in the following, as in the next lemma.

\begin{lemma}
\label{lem:tangent-covariant}
With notation as above, suppose that $K$ is a Hopf subalgebra of $U$.
Let $T \subset U$ be a finite-dimensional subspace satisfying the properties
\[
1 \in T, \quad
\Delta(T) \subset T \otimes U, \quad
K \triangleright T \subset T.
\]
Then mapping $T$ to $\calB^\circ$ gives rise to a left-covariant FODC over $\calB$.
\end{lemma}

\begin{proof}
As described above, we can consider any $u \in U$ as a linear functional on $\calB$ by $b \mapsto b(u)$.
In this way we obtain a map from $U$ to the dual of $\calB$, which is not injective in general.
It is a general fact that the image of this map is in $\calB^\circ$, due to $U$ being a coalgebra.

Under this map, the unit $1 \in U$ corresponds to the counit $\varepsilon$ on $\calB$.
Furthermore, writing $\tilde{T}$ for the image of $T$ in $\calB^\circ$, the condition $\Delta(T) \subset T \otimes U$ translates to $\Delta(\tilde{T}) \subset \tilde{T} \otimes \calB^\circ$.

For the last condition, consider $k \in K$, $u \in U$ and $b \in \calB$.
Using the identity $x y = (x_{(1)} \triangleright y) x_{(2)}$ and the fact that $K$ is a Hopf subalgebra, we compute
\[
b(k u) = b((k_{(1)} \triangleright u) k_{(2)})
= \varepsilon(k_{(2)}) b(k_{(1)} \triangleright u)
= b(k \triangleright u).
\]
Here we used the fact that $\calB$ is invariant under $K$.
This shows that $k u = k \triangleright u$ as elements of $\calB^\circ$, which gives $K \tilde{T} \subset \tilde{T}$.
Then the result follows from \cref{prop:calculi-correspondence}.
\end{proof}

This formulation brings us in close contact with the corresponding properties satisfied by the equivariant quantizations of $\lien = \lien_1 \oplus \cdots \oplus \lien_r$.

\subsection{Main result}

In this section we discuss how to use our equivariant quantizations of $\lien$ to produce covariant FODCs on $\CqG$ and certain quantum flag manifolds.
We consider the setting of \cite{heko-homogeneous} and compare our results with \cite{heko}, where the authors introduce canonical covariant differential calculi on quantum \emph{irreducible} flag manifolds.

In order to do this we need to switch conventions, since these references work with right coideals and also use the opposite coproduct with respect to ours, that is
\[
\Delta^\op(K_i) = K_i \otimes K_i, \quad
\Delta^\op(E_i) = E_i \otimes K_i + 1 \otimes E_i, \quad
\Delta^\op(F_i) = F_i \otimes 1 + K_i^{-1} \otimes F_i.
\]
To switch to this setting, we consider the quantum analogue of the Chevalley involution, which is the Hopf algebra isomorphism $\omega: \Uqg \to \Uqg^{\mathrm{cop}}$ given by
\[
\omega(K_i) = K_i^{-1}, \quad
\omega(E_i) = F_i, \quad
\omega(F_i) = E_i.
\]
Here $\Uqg^{\mathrm{cop}}$ denotes $\Uqg$ with the opposite coproduct.
We also denote by $\Uqnm$ the quantization corresponding to the negative nilradical $\lien_-$, or more concretely the subalgebra of $\Uqg$ generated by the elements $\{ F_i \}_{i = 1}^r$.

\begin{lemma}
The map $\omega: \Uqg \to \Uqg^{\mathrm{cop}}$ sends left $\Uqg$-coideals of $\Uqn \subset \Uqg$ to right $\Uqg$-coideals of $\Uqnm \subset \Uqg^{\mathrm{cop}}$.
\end{lemma}

\begin{proof}
It is clear that $\omega$ sends left coideals to right coideals, since it is a coalgebra map.
Finally $\omega(\Uqn) = \Uqnm$ by definition of $\omega$.
\end{proof}

With these preparations in place, we can state our main result for this section.
We note that the restrictions on $\lieg$ are exactly those from \cref{thm:coideal-quantizations}, which guarantee that $\bbC \oplus \qnil$ is a left $\Uqg$-coideal (the general result would follow from \cref{conj:coideal}).

\begin{theorem}
\label{thm:covariant-fodc}
Let $\lieg$ be a complex semisimple Lie algebra of rank $r$, which does not contain components of type $F_4, E_7, E_8$.
For any increasing sequence $S_1 \subset \cdots \subset S_r$ of proper subsets of $\simpleroots$, consider the equivariant quantization
\[
\qnil = \qnil_1 \oplus \cdots \oplus \qnil_r
\]
of the positive nilradical as in \cref{thm:coideal-quantizations}.

Write $\tang_- = \bbC \oplus \omega(\qnil)$ and $\tang_+ = \bbC \oplus \omega(\qnil)^*$. Then:
\begin{enumerate}
\item $\tang_-$ and $\tang_+$ correspond to left-covariant FODCs on $\Cq[G]$.
\item They also give left-covariant FODCs on $\Cq[G / P_{S_k}]$ for any $k \in \{ 1, \cdots, r \}$.
\end{enumerate}
\end{theorem}

\begin{proof}
(1) By \cref{thm:coideal-quantizations} the finite-dimensional subspace $\bbC \oplus \qnil \subset \Uqn$ is a left $\Uqg$-coideal.
Hence $\tang_- = \bbC \oplus \omega(\qnil) \subset \Uqnm$ is a right $\Uqg$-coideal containing the unit.
By \cref{lem:tangent-covariant} this gives a left-covariant FODC over $\Cq[G] = \Uqg^\circ$ (with respect to the trivial Hopf subalgebra $K = \bbC$).
The situation is similar for $\tang_+ = \bbC \oplus \omega(\qnil)^* \subset \Uqb$, where in this case we use the fact that the coproduct is a $*$-homomorphism.

(2) We begin by recalling the relevant classical facts concerning the decomposition $\lien = \lien_1 \oplus \cdots \oplus \lien_r$, as discussed in \cref{sec:decompositions}.
We write $\liel_k = \liel_{S_k}$ for the Levi factor corresponding to the subset $S_k \subset \simpleroots$ and recall that $\lien_{k - 1} \subset \liel_k$.
Since $\liel_1 \subset \cdots \subset \liel_r$, we find that:
\begin{itemize}
\item for $j \geq k$, we have that $\lien_j$ is a $\liel_k$-module,
\item for $j < k$, we have the inclusion $\lien_j \subset \liel_k$.
\end{itemize}
The situation is similar in the quantum setting.
We have $U_q(\liel_1) \subset \cdots \subset U_q(\liel_r)$, since the quantized Levi factor $\UqlS$ is generated by the elements $\{ K_i^{\pm} \}_{i = 1}^r$ and $\{ E_i, F_i \}_{i \in S}$.
Keeping in mind the construction of $\qnil_k$ from \cref{prop:equivariant-quantizations}, we find that:
\begin{itemize}
\item for $j \geq k$, we have that $\qnil_j$ is a $\Uqlk$-module under $\triangleright$,
\item for $j < k$, we have the inclusion $\qnil_j \subset \Uqlk$.
\end{itemize}

Now consider the quantum flag manifold $\calB = \Cq[G / P_{S_k}]$, which is defined as the subalgebra of $\calA = \Cq[G]$ invariant under the Hopf subalgebra $K = \Uqlk$.
As observed above, for $j \geq k$ we have $\Uqlk \triangleright \qnil_j \subseteq \qnil_j$.
On the other hand, for $j < k$ we have $\qnil_j \equiv 0$ as linear functionals on $\calB$, since $\qnil_j \subset \Uqlk$ and $\varepsilon(\qnil_j) = 0$.
It follows that the image of $\bbC \oplus \qnil$ in $\calB^\circ$ is invariant under the left-adjoint action of $K = \Uqlk$.
The same is true for $\tang_-$, since $\omega$ is an algebra map, and for $\tang_+$, since $\Uqlk$ is a Hopf $*$-subalgebra (and we have $x \triangleright y^* = (S(x)^* \triangleright y)^*$).
Then $\tang_-$ and $\tang_+$ give rise to left-covariant FODCs on $\Cq[G / P_{S_k}]$ by \cref{lem:tangent-covariant}.
\end{proof}

Finally we want to compare this construction with the results of \cite{heko}, which we briefly summarize.
For any quantum \emph{irreducible} flag manifold $\bbC_q[G / P_S]$, the authors construct two left-covariant (irreducible) FODCs $\Gamma_\del$ and $\Gamma_{\delbar}$, corresponding classically to the holomorphic and anti-holomorphic parts of the differential calculus over $G / P_S$.
We recall that irreducible flag manifolds arise from \emph{cominuscule} parabolic subalgebras, which correspond to subsets $S = \simpleroots \backslash \{\alpha_s\}$ with $\alpha_s$ appearing with coefficient \emph{one} in the highest root of $\lieg$.

Hence we want to consider the case of $\lieg$ with an increasing sequence $S_1 \subset \cdots \subset S_r$ such that $S_r = \simpleroots \backslash \{\alpha_{i_r}\}$ and $\alpha_{i_r}$ appears with coefficient one in the highest root of $\lieg$ (let us also note in passing that $F_4$, $E_7$, $E_8$ do not admit any irreducible flag manifolds).

\begin{proposition}
\label{prop:comparison-heckenberger-kolb}
For any increasing sequence $S_1 \subset \cdots S_r$ as described above, consider the left-covariant FODCs $\tang_-$ and $\tang_+$ on $\bbC_q[G / P_{S_r}]$ constructed in \cref{thm:covariant-fodc}.
They coincide with the FODCs $\Gamma_\del$ and $\Gamma_{\delbar}$ of Heckenberger-Kolb.
\end{proposition}

\begin{proof}
Write $\calB = \bbC_q[G / P_{S_r}]$ and consider $\qnil = \qnil_1 \oplus \cdots \oplus \qnil_r$, corresponding to a sequence $S_1 \subset \cdots \subset S_r$ as described above.
We have $\qnil_1 \oplus \cdots \oplus \qnil_{r - 1} \subset U_q(\liel_{S_r})$, as already discussed in \cref{thm:covariant-fodc}. These elements vanish when mapped to $\calB^\circ$, since $\calB$ is defined by invariance under $U_q(\liel_{S_r})$. Hence $\qnil |_\calB = \qnil_r |_\calB$ and the same is true for $\omega(\qnil)$.

In the cominuscule case all the radical roots contain $\alpha_{i_r}$ with coefficient one, which means that the summand $\lien_r$ has only a component of degree one.
For its quantum counterpart $\qnil_r$, this means that is spanned by the quantum root vectors $\{ E_{\xi_1}, \cdots, E_{\xi_n} \}$ corresponding to the radical roots, as in \cref{prop:schubert-component}.
By applying $\omega$ to $\qnil_r$, we find that $\omega(\qnil_r)$ is spanned by the quantum root vectors $\{ F_{\xi_1}, \cdots, F_{\xi_n} \}$.
To check this, we can use the identity $\omega(T_i(X_\alpha)) = c_{i, \alpha} T_i(\omega(X_\alpha))$, which is valid for any weight element $X_\alpha \in \Uqg$ (for some coefficient $c_{i, \alpha}$ that depends only on $i$ and $\alpha$).
By \cite[Proposition 3.3(v)]{heko}, the quantum tangent space for the holomorphic part $\Gamma_\del$ of the FODC is $\lspan \{ F_{\xi_1}, \cdots, F_{\xi_n} \}$.

Next consider $\omega(\qnil_r)^* = \lspan \{ F_{\xi_1}^*, \cdots, F_{\xi_n}^* \}$.
The quantum tangent space for the anti-holomorphic part $\Gamma_{\delbar}$, which appears in \cite[Proposition 3.4(v)]{heko}, is given instead as $\lspan \{ E_{\xi_1}, \cdots, E_{\xi_n} \}$.
The two vector spaces appear different at first sight, and it is certainly not true that $F_{\xi_k}^*$ equals $E_{\xi_k}$ in $\Uqg$. However the two should coincide as elements of $\calB^\circ$, since in this case the Cartan part acts trivially.
Rather than proceeding this way, we can make use of \cite[Theorem 4.2]{kahler}, where it is shown that the FODC $\Gamma = \Gamma_\del \oplus \Gamma_{\delbar}$ is a $*$-calculus over $\calB$.
In this case the quantum tangent space must be invariant under $*$, see for instance \cite[Section 14.1.2, Proposition 6]{klsc}.
This gives the result.
\end{proof}

Hence the FODCs introduced here over the quantum flag manifolds $\bbC_q[G / P_{S_k}]$ generalize those of Heckenberger-Kolb beyond the irreducible case.
The study of further properties of these calculi is left for future investigations.
We only briefly mention that they should have classical dimension, since by specialization at $q = 1$ the quantum tangent spaces considered here should reduce to the classical tangent spaces.
Finally, we should also mention that a similar analysis has recently appeared in \cite{reamonn-ar} for the $A_r$ series.

\appendix

\section{Tensor product decompositions}
\label{sec:appendix-computations}

In this appendix we verify the property used in \cref{prop:coideal-degree-three} for certain tensor products decompositions.
Recall that, if $S \subset \simpleroots$ is a subset of the simple roots of $\lieg$, we denote by $\lieg_S$ the semisimple part of the Levi factor corresponding to $S$, in other words the Lie subalgebra determined by the Dynkin subdiagram corresponding to $S$.
In the case of $S = \simpleroots \backslash \{\alpha_s\}$, the nilradical $\lien_S$ corresponding to $S$ admits an $\bbN$-grading by $\alpha_s$, and we denote by $V_n$ its graded components.
The result concerns the multiplicity of $V_3$ inside $V_1^{\otimes 3}$.

\begin{figure}[h]
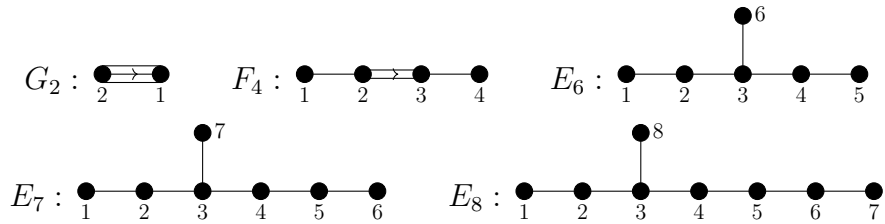

\begin{center}

\begin{tabular}{ccc}
$G_2:$ \dynkin[scale = 2.2, labels = {2, 1}] G2 & \quad
$F_4:$ \dynkin[scale = 2.2, labels = {1, 2, 3, 4}] F4 & \quad
$E_6:$ \dynkin[scale = 2.2, labels = {1, 6, 2, 3, 4, 5}] E6
\end{tabular}

\begin{tabular}{cc}
$E_7:$  \dynkin[scale = 2.2, labels = {1, 7, 2, 3, 4, 5, 6}] E7 & \quad
$E_8:$ \dynkin[scale = 2.2, labels = {1, 8, 2, 3, 4, 5, 6, 7}] E8
\end{tabular}

\end{center}
\caption{The exceptional simple Lie algebras and the numbering we adopt.}
\label{fig:dynkin-numbering}
\end{figure}

\begin{lemma}
\label{lem:degree-three-multiplicity}
Let $\lieg$ be a complex semisimple Lie algebra.
Consider $S = \simpleroots \backslash \{\alpha_s\}$, where $\alpha_s$ has coefficient $\geq 3$ in the highest root of $\lieg$.
Let $V_1$ and $V_3$ be the $\lieg_S$-modules corresponding to the components of degree one and three with respect to the grading by $\alpha_s$.

Then the simple component $V_3$ has multiplicity two in $V_1^{\otimes 3}$.
\end{lemma}

\begin{proof}
We can reduce to the simple case, since the root system of a semisimple Lie algebra is the disjoint union of the root systems of the simple components. Then $\alpha_s$ belongs to exactly one component, with the other components acting trivially on $V_1$ and $V_3$.

A simple root $\alpha_s$ can appear with coefficient $\geq 3$ in the highest root of $\lieg$ only for the exceptional Lie algebras.
The corresponding highest roots $\theta_\lieg$ are
\[
\begin{split}
\theta_{E_6} & = \alpha_1 + 2 \alpha_2 + 3 \alpha_3 + 2 \alpha_4 + \alpha_5 + 2 \alpha_6, \\
\theta_{E_7} & = 2 \alpha_1 + 3 \alpha_2 + 4 \alpha_3 + 3 \alpha_4 + 2 \alpha_5 + \alpha_6 + 2 \alpha_7, \\
\theta_{E_8} & = 2 \alpha_1 + 4 \alpha_2 + 6 \alpha_3 + 5 \alpha_4 + 4 \alpha_5 + 3 \alpha_6 + 2 \alpha_7 + 3 \alpha_8, \\
\theta_{F_4} & = 2 \alpha_1 + 3 \alpha_2 + 4 \alpha_3 + 2 \alpha_4, \\
\theta_{G_2} & = 3 \alpha_1 + 2 \alpha_2.
\end{split}
\]
Here the numbering for the simple roots is as depicted in \cref{fig:dynkin-numbering}.

For each of these cases, we consider all subsets $S = \simpleroots \backslash \{\alpha_s\}$ such that $\alpha_s$ appears with coefficient $\geq 3$ in $\theta_\lieg$.
Write $\lambda_1$ and $\lambda_3$ for the highest weights of the $\lieg_S$-modules $V_1$ and $V_3$, corresponding to the components of degrees one and three.
Our goal is to verify that $V(\lambda_3)$ appears with multiplicity two in the tensor product $V(\lambda_1)^{\otimes 3}$.

To minimize potential mistakes in the verification, we have used the Mathematica package LieART \cite{lieart}, which uses the same numbering for the Dynkin diagrams we have considered here.
The code for this verification is available with the arXiv submission.

Below we tabulate the relevant data for each exceptional Lie algebra $\lieg$.
In the first row we write the simple root $\alpha_s$ under consideration, together with the semisimple Lie algebra $\lieg_S$ corresponding to $S = \simpleroots \backslash \{\alpha_s\}$.
In the second row we write the concrete identification between the simple roots and the fixed numbering for the Dynkin diagram of $\lieg_S$ as given in \cref{fig:dynkin-numbering}.
In the third and fourth row we write the highest weights $\lambda_1$ and $\lambda_3$ for the components of degree one and three.
We express them both in terms of the simple roots of $\lieg$ and in terms of the fundamental weights of the semisimple Lie algebra $\lieg_S$.

For $\lieg = G_2$ we have one case to consider.

{\renewcommand{\arraystretch}{1.4}
\begin{center}
\begin{tabular}{|l|} 
\hline
$\alpha_s = \alpha_1$ and $\lieg_S = A_1$ \\
\hline \hline
$(\alpha_2) \ \leftrightarrow \ (\alpha_1)$ \\
\hline
$\lambda_1 = \alpha_1 + \alpha_2 \equiv \omega_1$ \\
\hline
$\lambda_3 = 3 \alpha_1 + 2 \alpha_2 \equiv \omega_1$ \\
\hline
\end{tabular}
\end{center}}

For $\lieg = F_4$ we have two cases to consider.

{\renewcommand{\arraystretch}{1.4}
\begin{center}
\begin{tabular}{|l|} 
\hline
$\alpha_s = \alpha_2$ and $\lieg_S = A_1 \oplus A_2$ \\
\hline \hline
$(\alpha_1) \sqcup (\alpha_3, \alpha_4) \ \leftrightarrow \ (\alpha_1) \sqcup (\alpha_1, \alpha_2)$ \\
\hline
$\lambda_1 = \alpha_1 + \alpha_2 + 2 \alpha_3 + 2 \alpha_4 \equiv (\omega_1, 2 \omega_2)$ \\
\hline
$\lambda_3 = 2 \alpha_1 + 3 \alpha_2 + 4 \alpha_3 + 2 \alpha_4 \equiv (\omega_1, 0)$ \\
\hline
\end{tabular}
\end{center}}

{\renewcommand{\arraystretch}{1.4}
\begin{center}
\begin{tabular}{|l|} 
\hline
$\alpha_s = \alpha_3$ and $\lieg_S = A_2 \oplus A_1$ \\
\hline \hline
$(\alpha_1, \alpha_2) \sqcup (\alpha_4) \ \leftrightarrow \ (\alpha_1, \alpha_2) \sqcup (\alpha_1)$ \\
\hline
$\lambda_1 = \alpha_1 + \alpha_2 + \alpha_3 + \alpha_4 \equiv (\omega_1, \omega_1)$ \\
\hline
$\lambda_3 = \alpha_1 + 2 \alpha_2 + 3 \alpha_3 + 2 \alpha_4 \equiv (0, \omega_1)$ \\
\hline
\end{tabular}
\end{center}}

For $\lieg = E_6$ we have one case to consider.

{\renewcommand{\arraystretch}{1.4}
\begin{center}
\begin{tabular}{|l|} 
\hline
$\alpha_s = \alpha_3$ and $\lieg_S = A_2 \oplus A_2 \oplus A_1$ \\
\hline \hline
$(\alpha_1, \alpha_2) \sqcup (\alpha_4, \alpha_5) \sqcup (\alpha_6) \ \leftrightarrow \ (\alpha_1, \alpha_2) \sqcup (\alpha_1, \alpha_2) \sqcup (\alpha_1)$ \\
\hline
$\lambda_1 = \alpha_1 + \alpha_2 + \alpha_3 + \alpha_4 + \alpha_5 + \alpha_6 \equiv (\omega_1, \omega_2, \omega_1)$ \\
\hline
$\lambda_3 = \alpha_1 + 2 \alpha_2 + 3 \alpha_3 + 2 \alpha_4 + \alpha_5 + 2 \alpha_6 \equiv (0, 0, \omega_1)$ \\
\hline
\end{tabular}
\end{center}}

For $\lieg = E_7$ we have three cases to consider.

{\renewcommand{\arraystretch}{1.4}
\begin{center}
\begin{tabular}{|l|} 
\hline
$\alpha_s = \alpha_2$ and $\lieg_S = A_1 \oplus A_5$ \\
\hline \hline
$(\alpha_1) \sqcup (\alpha_7, \alpha_3, \alpha_4, \alpha_5, \alpha_6) \ \leftrightarrow \ (\alpha_1) \sqcup (\alpha_1, \alpha_2, \alpha_3, \alpha_4, \alpha_5)$ \\
\hline
$\lambda_1 = \alpha_1 + \alpha_2 + 2 \alpha_3 + 2 \alpha_4 + 2 \alpha_5 + \alpha_6 + \alpha_7 \equiv (\omega_1, \omega_4)$ \\
\hline
$\lambda_3 = 2 \alpha_1 + 3 \alpha_2 + 4 \alpha_3 + 3 \alpha_4 + 2 \alpha_5 + \alpha_6 + 2 \alpha_7 \equiv (\omega_1, 0)$ \\
\hline
\end{tabular}
\end{center}}

{\renewcommand{\arraystretch}{1.4}
\begin{center}
\begin{tabular}{|l|} 
\hline
$\alpha_s = \alpha_3$ and $\lieg_S = A_2 \oplus A_3 \oplus A_1$ \\
\hline \hline
$(\alpha_1, \alpha_2) \sqcup (\alpha_4, \alpha_5, \alpha_6) \sqcup (\alpha_7) \ \leftrightarrow \ (\alpha_1, \alpha_2) \sqcup (\alpha_1, \alpha_2, \alpha_3) \sqcup (\alpha_1)$ \\
\hline
$\lambda_1 = \alpha_1 + \alpha_2 + \alpha_3 + \alpha_4 + \alpha_5 + \alpha_6 + \alpha_7 \equiv (\omega_1, \omega_3, \omega_1)$ \\
\hline
$\lambda_3 = \alpha_1 + 2 \alpha_2 + 3 \alpha_3 + 3 \alpha_4 + 2 \alpha_5 + \alpha_6 + 2 \alpha_7 \equiv (0, \omega_1, \omega_1)$ \\
\hline
\end{tabular}
\end{center}}

{\renewcommand{\arraystretch}{1.4}
\begin{center}
\begin{tabular}{|l|} 
\hline
$\alpha_s = \alpha_4$ and $\lieg_S = A_4 \oplus A_2$ \\
\hline \hline
$(\alpha_1, \alpha_2, \alpha_3, \alpha_7) \sqcup (\alpha_5, \alpha_6) \ \leftrightarrow \ (\alpha_1, \alpha_2, \alpha_3, \alpha_4) \sqcup (\alpha_1, \alpha_2)$ \\
\hline
$\lambda_1 = \alpha_1 + 2 \alpha_2 + 2 \alpha_3 + \alpha_4 + \alpha_5 + \alpha_6 + \alpha_7 \equiv (\omega_2, \omega_2)$ \\
\hline
$\lambda_3 = 2 \alpha_1 + 3 \alpha_2 + 4 \alpha_3 + 3 \alpha_4 + 2 \alpha_5 + \alpha_6 + 2 \alpha_7 \equiv (\omega_1, 0)$ \\
\hline
\end{tabular}
\end{center}}

For $\lieg = E_8$ we have six cases to consider.

{\renewcommand{\arraystretch}{1.4}
\begin{center}
\begin{tabular}{|l|} 
\hline
$\alpha_s = \alpha_2$ and $\lieg_S = A_1 \oplus A_6$ \\
\hline \hline
$(\alpha_1) \sqcup (\alpha_8, \alpha_3, \alpha_4, \alpha_5, \alpha_6, \alpha_7) \ \leftrightarrow \ (\alpha_1) \sqcup (\alpha_1, \alpha_2, \alpha_3, \alpha_4, \alpha_5, \alpha_6)$ \\
\hline
$\lambda_1 = \alpha_1 + \alpha_2 + 2 \alpha_3 + 2 \alpha_4 + 2 \alpha_5 + 2 \alpha_6 + \alpha_7 + \alpha_8 \equiv (\omega_1, \omega_5)$ \\
\hline
$\lambda_3 = 2 \alpha_1 + 3 \alpha_2 + 5 \alpha_3 + 4 \alpha_4 + 3 \alpha_5 + 2 \alpha_6 + \alpha_7 + 3 \alpha_8 \equiv (\omega_1, \omega_1)$ \\
\hline
\end{tabular}
\end{center}}

{\renewcommand{\arraystretch}{1.4}
\begin{center}
\begin{tabular}{|l|} 
\hline
$\alpha_s = \alpha_3$ and $\lieg_S = A_2 \oplus A_4 \oplus A_1$ \\
\hline \hline
$(\alpha_1, \alpha_2) \sqcup (\alpha_4, \alpha_5, \alpha_6, \alpha_7) \sqcup (\alpha_8) \ \leftrightarrow \ (\alpha_1, \alpha_2) \sqcup (\alpha_1, \alpha_2, \alpha_3, \alpha_4) \sqcup (\alpha_1)$ \\
\hline
$\lambda_1 = \alpha_1 + \alpha_2 + \alpha_3 + \alpha_4 + \alpha_5 + \alpha_6 + \alpha_7 + \alpha_8 \equiv (\omega_1, \omega_4, \omega_1)$ \\
\hline
$\lambda_3 = \alpha_1 + 2 \alpha_2 + 3 \alpha_3 + 3 \alpha_4 + 3 \alpha_5 + 2 \alpha_6 + \alpha_7 + 2 \alpha_8 \equiv (0, \omega_2, \omega_1)$ \\
\hline
\end{tabular}
\end{center}}

{\renewcommand{\arraystretch}{1.4}
\begin{center}
\begin{tabular}{|l|} 
\hline
$\alpha_s = \alpha_4$ and $\lieg_S = A_4 \oplus A_3$ \\
\hline \hline
$(\alpha_1, \alpha_2, \alpha_3, \alpha_8) \sqcup (\alpha_5, \alpha_6, \alpha_7) \ \leftrightarrow \ (\alpha_1, \alpha_2, \alpha_3, \alpha_4) \sqcup (\alpha_1, \alpha_2, \alpha_3)$ \\
\hline
$\lambda_1 = \alpha_1 + 2 \alpha_2 + 2 \alpha_3 + \alpha_4 + \alpha_5 + \alpha_6 + \alpha_7 + \alpha_8 \equiv (\omega_2, \omega_3)$ \\
\hline
$\lambda_3 = 2 \alpha_1 + 3 \alpha_2 + 4 \alpha_3 + 3 \alpha_4 + 3 \alpha_5 + 2 \alpha_6 + \alpha_7 + 2 \alpha_8 \equiv (\omega_1, \omega_1)$ \\
\hline
\end{tabular}
\end{center}}

{\renewcommand{\arraystretch}{1.4}
\begin{center}
\begin{tabular}{|l|} 
\hline
$\alpha_s = \alpha_6$ and $\lieg_S = E_6 \oplus A_1$ \\
\hline \hline
$(\alpha_1, \alpha_2, \alpha_3, \alpha_4, \alpha_5, \alpha_8) \sqcup (\alpha_7) \ \leftrightarrow \ (\alpha_1, \alpha_2, \alpha_3, \alpha_4, \alpha_5, \alpha_6) \sqcup (\alpha_1)$ \\
\hline
$\lambda_1 = 2 \alpha_1 + 3 \alpha_2 + 4 \alpha_3 + 3 \alpha_4 + 2 \alpha_5 + \alpha_6 + \alpha_7 + 2 \alpha_8 \equiv (\omega_1, \omega_1)$ \\
\hline
$\lambda_3 = 2 \alpha_1 + 4 \alpha_2 + 6 \alpha_3 + 5 \alpha_4 + 4 \alpha_5 + 3 \alpha_6 + 2 \alpha_7 + 3 \alpha_8 \equiv (0, \omega_1)$ \\
\hline
\end{tabular}
\end{center}}

{\renewcommand{\arraystretch}{1.4}
\begin{center}
\begin{tabular}{|l|} 
\hline
$\alpha_s = \alpha_8$ and $\lieg_S = A_7$ \\
\hline \hline
$(\alpha_1, \alpha_2, \alpha_3, \alpha_4, \alpha_5, \alpha_6, \alpha_7) \ \leftrightarrow \ (\alpha_1, \alpha_2, \alpha_3, \alpha_4, \alpha_5, \alpha_6, \alpha_7)$ \\
\hline
$\lambda_1 = \alpha_1 + 2 \alpha_2 + 3 \alpha_3 + 3 \alpha_4 + 3 \alpha_5 + 2 \alpha_6 + \alpha_7 + \alpha_8 \equiv \omega_5$ \\
\hline
$\lambda_3 = 2 \alpha_1 + 4 \alpha_2 + 6 \alpha_3 + 5 \alpha_4 + 4 \alpha_5 + 3 \alpha_6 + 2 \alpha_7 + 3 \alpha_8 \equiv \omega_7$ \\
\hline
\end{tabular}
\end{center}}

In all cases listed above, we have that $V(\lambda_3)$ appears with multiplicity two in $V(\lambda_1)^{\otimes 3}$.
\end{proof}

At the moment we do not have a conceptual reason as for why this property should be expected.
We note that this could be reformulated as a verification on certain decompositions of the positive roots (which is perhaps easier to verify), since each simple component $V(\nu)$ occurring in $V(\lambda) \otimes V(\mu)$ is of the form $\nu = \lambda + \mu^\prime$, with $\mu^\prime$ a weight of $V(\mu)$.

\end{document}